\documentclass[12pt, a4paper, parskip=half, abstracton]{scrartcl}

\usepackage{array}
\usepackage{marginnote}
\usepackage{xcolor}
\usepackage{amscd,amssymb,amsfonts,amsmath,latexsym,amsthm}
\usepackage{hyperref}
\usepackage[all,cmtip]{xy}
\usepackage{mathrsfs}
\usepackage{graphicx}
\usepackage{bm} 
\usepackage[nameinlink, capitalise]{cleveref}

\usepackage{etoolbox}

\def\hB{\hspace*{\fill}$\qed$}

\usepackage[nottoc]{tocbibind} 

\usepackage{defs_pp1}
\usepackage{slashed}
\usepackage[utf8]{inputenc}
\usepackage{microtype}
\usepackage[english]{babel}
\usepackage{mathtools}
\usepackage{bm}
\usepackage{esvect}

\title{The coarse Pimsner-Voiculescu sequence}
\author{
Ulrich Bunke\thanks{Fakult{\"a}t f{\"u}r Mathematik,
Universit{\"a}t Regensburg,
93040 Regensburg,
ulrich.bunke@mathematik.uni-regensburg.de} 
}

\numberwithin{equation}{section}
\newtheorem{theorem}{Theorem}[section] 
\newtheorem{prop}[theorem]{Proposition}
\newtheorem{lem}[theorem]{Lemma}

\newtheorem{ddd}[theorem]{Definition}
\newtheorem{kor}[theorem]{Corollary}

\theoremstyle{remark}
\theoremstyle{definition}

\newtheorem{ex}[theorem]{Example}
\newtheorem{rem}[theorem]{Remark}

\renewcommand{\loc}{\mathrm{loc}}
\newcommand{\free}{\mathrm{free}}
\newcommand{\PV}{\mathbf{PV}}

\newcommand{\bX}{\mathbf{X}}

\newcommand{\UBC}{\mathbf{UBC}}

\newcommand{\Yo}{\mathrm{Yo}}

\newcommand{\Res}{\mathrm{Res}}

\newcommand{\Orb}{\mathbf{Orb}}

\newcommand{\Hilb}{\mathbf{Hilb}}

\newcommand{\BC}{\mathbf{BC}}

\newcommand{\Fin}{\mathbf{Fin}}

\newcommand{\bM}{\mathbf{M}}

\newcommand{\bA}{{\mathbf{A}}}

\newcommand{\cO}{{\mathcal{O}}}

\newcommand{\cE}{{\mathcal{E}}}

\newcommand{\bP}{\mathbf{P}}

\newcommand{\KK}{\mathrm{KK}}

\newcommand{\cone}{\mathrm{cone}}

\newcommand{\nCcat}{C^{*}\mathbf{Cat}^{\mathrm{nu}}}

\newcommand{\Kcat}{\mathrm{K}^{C^{*}\mathrm{Cat}}}

 \newcommand{\coass}{\mathrm{coass}}

\newcommand{\disc}{\mathrm{disc}}

\begin{document}
\maketitle  
\begin{abstract}
We derive the Pimsner-Voiculescu sequence calculating the $K$-theory of a $C^{*}$-algebra with $\Z$-action using constructions with equivariant coarse $K$-homology theory. 
We then investigate to which extend this  idea extends to more general equivariant coarse homology theories.   
\end{abstract}
\tableofcontents

 \setcounter{tocdepth}{5}

\section{Introduction}

If $A$ is a $C^{*}$-algebra with an action of the group $\Z$, then the Pimsner-Voiculescu sequence \cite{pv}, \cite[10.2.1]{blackadar} is  the long exact sequence of $K$-theory groups \begin{equation}\label{asdfasdfadfdffs}
K_{*+1}(A\rtimes \Z)\to K_{*}(\Res^{\Z}(A))\stackrel{1-\sigma_{*}}{\to}
 K_{*}(\Res^{\Z}(A))\to K_{*}(A\rtimes \Z)\ ,
\end{equation}
 where  $\Res^{\Z}$ denotes the operation of forgetting the $\Z$-action, and 
$\sigma_{*}$ is induced by the $\Z$-action on $A$ via functoriality. 
It is the long exact sequence of homotopy groups 
associated to a fibre sequence  of $K$-theory spectra \begin{equation}\label{adfasdfadsfdasf}
\xymatrix{K(\Res^{\Z}(A))\ar[r]\ar[d]&K(A\rtimes \Z)\ar[d]\\0\ar[r]& \Sigma K(\Res^{\Z}(A))}\ .
\end{equation}
  
The classical arguments derive  such a fibre sequence of spectra from a suitably designed short exact sequence of $C^{*}$-algebras.  Thereby different proofs work with different sequences. 
They all give the long exact sequence \eqref{asdfasdfadfdffs} of homotopy groups,   but the question whether the resulting fibre sequences of spectra are equivalent has not been discussed in the literature so far.

The Baum-Connes conjecture with coefficients for the amenable group $\Z$  is known to hold and asserts that
 the assembly map   \begin{equation} \label{eqwfwefwdwe} \mu^{Kasp}_{\Fin,A}:\KK^{\Z}(C_{0}(\R),A)\to K(A\rtimes \Z)\end{equation}
introduced by Kasparov \cite{kasparovinvent} (see  \cite[13.9]{bel-paschke} for the spectrum level version)
 is an equivalence. 
 Accepting to use the Baum-Connes conjecture with coefficients for $\Z$
 a fibre sequence of the form \eqref{adfasdfadsfdasf} can be derived    in the following two ways.  
 
Observe that   $\KK^{\Z}(C_{0}(-),A)$ is a  locally finite $\Z$-equivariant homology theory. Using the  
 $\Z$-CW-structure of $\R$ with precisely  two    $\Z$-cells in   dimension $0$  and $1$
 and the identification $\KK^{\Z}(C_{0}(\Z),A)\simeq K(\Res^{\Z}(A))$ (use \cite[Cor. 1.23]{KKG})
 we get a fibre sequence   of the form \eqref{adfasdfadsfdasf}.
 
Alternatively,  
to $A$ one can associate a Davis-L\"uck type functor
$K^{\Z}_{A}:\Z\Orb\to \Sp$ \cite{kranz} (see also  \cite[Def. 16.15]{bel-paschke}) whose values on the $\Z$-orbits $\Z$ and $*$ are given by \begin{equation}\label{ewfewfawefwead}
K(\Res^{\Z}(A))\simeq K^{\Z}_{A}(\Z)\ , \quad K(A\rtimes \Z)\simeq K_{A}^{\Z}(*)\ .
\end{equation}   
If we equip $K(\Res^{\Z}(A))$ with the $\Z$-action induced by functoriality from the $\Z$-action on $A$, and $ K^{\Z}_{A}(\Z)$ with the $\Z$-action induced by functoriality from the $\Z$-action on the orbit $\Z$, then  the first equivalence   in \eqref{ewfewfawefwead}  is $\Z$-equivariant. The Davis-L\"uck assembly map    (see \eqref{wefqwedqwdeewdeqwd}  below)   turns out to be equivalent      to the map
\begin{equation}\label{qfqwefewdqwed}
\colim_{B\Z} K(A) \to K(A\rtimes \Z)\ .
\end{equation}  Since  Kasparov's assembly map \eqref{eqwfwefwdwe}  and the Davis-L\"uck assembly map are isomorphic on the level of homotopy groups \cite{kranz}, if one of them is an equivalence, then so is the other. Therefore  the Baum-Connes conjecture with coefficients for $\Z$, stating that Kasparov's assembly map is an equivalence,  also implies  that  \eqref{qfqwefewdqwed} is an equivalence.  The usual cofibre sequence calculating the  coinvariants of a $\Z$-object in a stable $\infty$-category   specializes to 
 $$K(\Res^{\Z}(A))\stackrel{1-\sigma}{\to} K(\Res^{\Z}(A)) \to K(A\rtimes \Z)\ ,$$ (we write the square  \eqref{adfasdfadsfdasf} in this form in order to be able to highlight the map $1-\sigma$),
 where $\sigma$ denotes the action of the generator of $\Z$ on $K(\Res^{\Z}(A))$. On the level of homotopy groups this cofibre sequence 
  induces the PV-sequence \eqref{asdfasdfadfdffs} 
 including the calculation of the map in the middle.

One outcome of  this note is  an alternative construction of   a  square of the form \eqref{adfasdfadsfdasf} in terms of the $\Z$-equivariant coarse the $K$-homology theory. In particular, in our construction the maps  involved in this square  acquire  a clear geometric interpretation. 

The symmetric monoidal category $(\Z\BC,\otimes)$ of $\Z$-bornological coarse spaces and the notion of a $\Z$-equivariant coarse homology theory have been 
introduced in \cite{equicoarse}.  Examples of $\Z$-bornological coarse spaces are $\Z_{min,min}$ and $\Z_{can,min}$ given by the group $\Z$ with the minimal bornology and  the  minimal or canonical coarse structures.
The additional structure of transfers was introduced in \cite{coarsetrans}.
 For every $\Z$-equivariant coarse homology theory $E^{\Z}:\Z\BC\to \bM$ with transfers and $\Z$-bornological coarse space $X$  in Section \ref{wkgopgsgfdg}
we will construct a commutative square  
 \begin{equation}\label{qwdqwewqwdqweddwdwdwd1eee} \xymatrix{ E^{\Z}( X\otimes \Z_{min,min})\ar[r]^-{\iota}\ar[d]& E^{  \Z}(X\otimes \Z_{can,min})\ar[d]^{\tr}\\0\ar[r]& E^{\Z}(X\otimes \Z_{can,min}\otimes \Z_{min,min})}
\end{equation} 
which we call the coarse PV-square associated to $E^{\Z}$ and $X$. 
The morphism $\iota$ in  \eqref{qwdqwewqwdqweddwdwdwd1eee} is induced by the   morphism
of bornological coarse spaces $\Z_{min,min}\to \Z_{can,min}$ given by the identity of the underlying sets, and the morphism
$\tr$ is the transfer morphism  induced by the coarse covering 
$\Z_{min,min}\to *$. The filler of the square is given by a simple
argument using the properties of   coarse homology theories.  

We then study conditions on $E^{\Z}$ and $X$ which imply that the square 
\eqref{qwdqwewqwdqweddwdwdwd1eee} cartesian. 


In order to formulate one such condition, for any $\Z$-equivariant coarse homology theory $F^{\Z}:\Z\BC\to \bM$ 
 we form the functor $$HF^{\Z}  :\Z\Orb\to \bM\ , \quad S\mapsto F^{\Z}(S_{min,max}\otimes \Z_{can,min})$$
 and consider   
 the Davis-L\"uck   assembly map \begin{equation}\label{wefqwedqwdeewdeqwd}
 \mu^{DL}_{HF^{\Z} }:\colim_{\Z_{\Fin}\Orb}HF^{\Z}\to HF^{\Z}(*)
\end{equation} 
  for  the family $\Fin$ of finite subgroups. 

We  
 form the $\Z$-equivariant coarse homology theory
$E_{X,c}^{\Z}:\Z\BC\to \bM$ given by the continuous approximation  of thq    $E^{\Z}(X\otimes -):\Z\BC\to \bM$ of $E^{\Z}$ by $X$.   \begin{theorem}[{Theorem \ref{weijtgowegferwerwgwergre}}]\label{gwjreogrgrfwerefwerf}
 Assume:
\begin{enumerate}
\item\label{sdvsdfvfewcd} $X$  has the minimal bornology and is discrete.
\item\label{wetkogierfrewfwef} $E^{\Z}$  is strong and strongly additive.
\end{enumerate}
 Then the PV-square \eqref{qwdqwewqwdqwedefe1eee}   is cartesian if and only if $\mu_{HE^{\Z}_{X,c}}^{DL}$ is an equivalence.
\end{theorem}
We refer to Section \ref{wkgopgsgfdg} for an explanation  and references concerning  the additional properties of $\Z$-equivariant coarse homology theories appearing in the above statement.
The  Condition \ref{gwjreogrgrfwerefwerf}.\ref{wetkogierfrewfwef} on $E^{\Z}$ is satisfied for  many examples of equivariant coarse homology theories, e.g. the coarse topological $K$-theory \eqref{qwedqwedqwedqwd} with coefficients in a $C^{*}$-category with $\Z$-action, or the coarse algebraic $K$-homology   associated to  an additive  category or a left-exact $\infty$-category with $\Z$-action. The first is the main example for the present paper, and we refer to   Example \ref{wrthokwegfwerwrf} for detailed references for the second case and to \cite{unik} for the third. 
The condition on  $\mu_{HE^{\Z}_{X,c}}^{DL}$
 is complicated  and not always satisfied, see Example \ref{wrthokwegfwerwrf}. 

At a first glance the Condition \ref{gwjreogrgrfwerefwerf}.\ref{sdvsdfvfewcd} on $X$   is very restrictive, but using that the corners of 
 the square \eqref{qwdqwewqwdqweddwdwdwd1eee} are $\Z$-equivariant coarse homology theories in the variable $X$
 one can extend the range of $\Z$-bornological spaces $X$ for which this square is known to be cartesian considerably.
Let $\Yo^{s}:\Z\BC\to \Z\Sp\cX$ be the universal $\Z$-equivariant coarse homology theory.
Then the property that \eqref{qwdqwewqwdqweddwdwdwd1eee} is cartesian only depends on the image 
$\Yo^{s}_{\loc}(X):=\ell(\Yo^{s}(X))$ of $\Yo^{s}(X)$ under the localization $\ell:\Z\Sp\cX\to \Z\Sp\cX_{\loc}$ of $\Z\Sp\cX$ at the three 
$\Z$-equivariant coarse homology theories $E^{\Z}( -\otimes \Z_{min,min})$, $E^{\Z}(- \otimes \Z_{can,min})$ and
$E^{\Z}( -\otimes \Z_{can,min}\otimes \Z_{min,min})$ appearing at the corners of \eqref{qwdqwewqwdqweddwdwdwd1eee}.
In Definition \ref{eiojgwergrefwerf} we introduce $\Z\Sp\cX_{\loc}\langle DL\rangle$ as the localizing subcategory of $\Z\Sp\cX_{\loc}$
generated by  $\Yo^{s}_{\loc}(X)$ for all $X$ in $\Z\BC$ which are discrete, have the minimal bornology, and are such that $\mu^{DL}_{HE^{\Z}_{X,c}}$ is an equivalence.
Then for $X$ in $\Z\BC$ we have the following consequence of Theorem \ref{gwjreogrgrfwerefwerf}.
 \begin{kor}[{Corollary \ref{wrthiojwergwergwerrwfg}}] \label{erthokeorpthertgergregretg}Assume that $E^{\Z}$  is strong and strongly additive.
 If $\Yo^{s}_{\loc}(X)\in \Z\Sp\cX_{\loc}\langle DL\rangle$, then the PV-square \eqref{qwdqwewqwdqwedefe1eee}   is cartesian.
 \end{kor}
 
There are many $\Z$-bornological coarse spaces $X$ 
which are not necessarily discrete or have the minimal bornology 
but still  satisfy $\Yo^{s}_{\loc}(X) \in \Z\Sp\cX_{\loc}\langle DL\rangle$.
It is even not clear that for such spaces 
the Davis-L\"uck assembly map $\mu^{DL}_{HE^{\Z}_{X,c}}$ is an equivalence, but nevertheless  the PV-square \eqref{qwdqwewqwdqwedefe1eee}   is cartesian. 


In order to connect the square \eqref{qwdqwewqwdqweddwdwdwd1eee}  with the classical PV-sequence \eqref{adfasdfadsfdasf} for a $\Z$-$C^{*}$-algebra $A$
we take  $E^{\Z}=K\cX_{\Hilb_{c}(A)}^{\Z}$, the coarse algebraic $K$-homology \eqref{qwedqwedqwedqwd} with coefficients in the $C^{*}$-category $\Hilb_{c}(A)$ of Hilbert $A$-modules and compact operators \cite{coarsek}, and  $X=*$. 
Then $E^{\Z}$ and $X$   satisfy the assumptions of Theorem \ref{gwjreogrgrfwerefwerf}.
 We further employ the fact that  group $\Z$ satisfies the Baum-Connes conjecture with coefficients in order to verify  that $\mu^{DL}_{HK\cX_{\bC}^{\Z}} $ is an equivalence.  In view of   the obvious equivalence
 $\mu^{DL}_{HK\cX_{\bC}^{\Z}} \simeq \mu^{DL}_{HK\cX_{\bC,*,c}^{\Z}} $, 
    by Theorem \ref{gwjreogrgrfwerefwerf}
the  coarse PV-square  
 \begin{equation}\label{qwdqwewqrwdqweddwdwdwd1eee} \xymatrix{ K\cX^{\Z}_{\Hilb_{c}(A)}(  \Z_{min,min})\ar[r]^-{\iota}\ar[d]&  K\cX_{\Hilb_{c}(A)}^{  \Z}(  \Z_{can,min})\ar[d]^{\tr}\\0\ar[r]&  K\cX_{\Hilb_{c}(A)}^{\Z}(  \Z_{can,min}\otimes \Z_{min,min})}
\end{equation} 
is cartesian. In  Proposition 
\ref{eorkjgwegreegrwegr9} we explain that this square  is 
 equivalent to a square of the form \eqref{adfasdfadsfdasf}. 
 We further determine the boundary map 
 explicitly.
We thus get a new  construction of a Pimsner-Voiculescu sequence. 

 

%
%

Note that $\Hilb_{c}(A)$ for a $\Z$-$C^{*}$-algebra $A$ is just a particular  example of a $C^{*}$-category with a strict $\Z$-action. More generally, 
 if   $\bC$ is any $C^{*}$-category with a strict $\Z$-action which  admits small orthogonal AV-sums, then   we have the $\Z$-equivariant coarse homology theory with transfers \begin{equation}\label{qwedqwedqwedqwd}K\cX_{\bC}^{\Z}:\Z\BC\to \Sp\end{equation}
 constructed in \cite{coarsek}.   For every 
$X$ in $\Z\BC$,  by specializing \eqref{qwdqwewqwdqweddwdwdwd1eee}, we  obtain the 
coarse PV-square 
 \begin{equation}\label{qwdqwerrwqwdqweddwdwdwd1eee} \xymatrix{ K\cX_{\bC}^{\Z}( X\otimes \Z_{min,min})\ar[r]^-{\iota}\ar[d]& K\cX_{\bC}^{\Z} (X\otimes \Z_{can,min})\ar[d]^{\tr}\\0\ar[r]& K\cX_{\bC}^{\Z}(X\otimes \Z_{can,min}\otimes \Z_{min,min})}\ .
\end{equation} 
If $X $ is discrete and has the minimal bornology, i.e., $X\cong Y_{min,min}$ for some $\Z$-set $Y$, then 
  by Proposition \ref{wiothgerththerh}  we have an equivalence of $\Z$-equivariant coarse homology theories 
$$K\cX_{\bC,X,c}^{\Z }\simeq K\cX_{\bC_{Y}}^{\Z} \ ,$$ 
where $ \bC_{Y}$ is a suitably defined $C^{*}$-category with $\Z$-action depending on $Y$. We can therefore reduce the problem of showing that  \eqref{qwdqwerrwqwdqweddwdwdwd1eee} is cartesian to the case of $X=*$ at the cost of modifying the coefficient $C^{*}$-category. Using the Baum-Connes conjecture with coefficients for $\Z$ we can then check that
 $\mu^{DL}_{HK\cX^{\Z}_{\bC_{Y}}}$ and hence
 $\mu^{DL}_{HK\cX^{\Z}_{\bC,X,c}}$ are equivalences so that $\Yo^{s}_{\loc}(X)\in \Z\Sp\cX_{\loc}\langle DL\rangle$ by Theorem \ref{gwjreogrgrfwerefwerf}.  
 
 Let $\Z\Sp\cX_{\loc}\langle \disc\rangle$ be the localizing subcategory of $\Z\Sp\cX_{\loc}$ generated by all $X$ in $\Z\BC$ which are discrete and have the minimal bornology.
 \begin{theorem}[{Theorem \ref{weitgowergerfrwferfw}}]
 Assume that $E^{\Z}=K\cX_{\bC}^{\Z}$. Then 
  $$\Z\Sp\cX_{\loc}\langle \disc\rangle\subseteq \Z\Sp\cX_{\loc}\langle DL\rangle\ .$$
  \end{theorem}

 For a general $X$ in $\Z\BC$ we then apply Corollary \ref{erthokeorpthertgergregretg}
in order to conclude that  \eqref{qwdqwerrwqwdqweddwdwdwd1eee} is cartesian provided $\Yo^{s}_{\loc}(X)\in \Z\Sp\cX_{\loc}\langle \disc\rangle$, see  Corollary \ref{qerigojoqrfewfqewfqf}. 

%
%
%
%
%
%

We now restrict to bornological coarse spaces with trivial $\Z$-action 
  and provide very general conditions on $X$ ensuring that $\Yo^{s}_{\loc}(X)\in \Z\Sp\cX_{\loc}\langle \disc\rangle$.
In order to state the result   we employ the  coarse assembly map 
\begin{equation}\label{fqwefwqedwqedqwedqewd}
\mu_{F,X}:F\cO^{\infty}\bP(X)\to F(X)
\end{equation} for a strong coarse homology theory $F:\BC\to \bM$ and  a bornological coarse space $X$ which has been introduced in \cite[Def. 9.7]{ass}. Recall the notion of weakly finite asymptotic dimension \cite[Def. 10.3]{ass} and bounded geometry \cite[Def. 7.77]{buen}. In the following we consider the functors $K\cX_{\bC}^{\Z}( -\otimes \Z_{min,min})$, $K\cX_{\bC}^{\Z} (-\otimes \Z_{can,min})$ and $K\cX_{\bC}^{\Z} (-\otimes \Z_{can,min}\otimes \Z_{min,min})$ from $ \BC$ to $\Sp$ as strong non-equivariant coarse homology theories.

\begin{theorem}[{Theorem \ref{werigosetrrgertgerwgw}}]\label{werigosetrgertgerwgw}Assume one of the following:
\begin{enumerate}
\item\label{ijtgoertgegertgwee} $X$ has weakly finite asymptotic dimension.
\item \label{ijtgoertgegertgwee1} $X$ has bounded geometry and the three coarse  assembly maps $\mu_{K\cX_{ \bC}(-\otimes \Z_{min,min}),X}$,
$\mu_{K\cX^{\Z}_{\bC}(-\otimes \Z_{can,min}),X}$ and
$\mu_{K\cX^{\Z}_{\bC}(-\otimes \Z_{can,min}\otimes \Z_{min,min}),X}$
are   equivalences.
\end{enumerate}
Then $\Yo^{s}_{\loc}(X)\in \Z\Sp\cX_{\loc}\langle \disc\rangle$ and hence
\eqref{qwdqwerrwqwdqweddwdwdwd1eee} is cartesian.
  \end{theorem}
 
 It could be true that the  coarse PV-square \eqref{qwdqwerrwqwdqweddwdwdwd1eee}
is cartesian  for all $X$ in $\BC$.
 
The Condition  \ref{werigosetrgertgerwgw}.\ref{ijtgoertgegertgwee} on $X$ in 
implies that the coarse assembly map $\mu_{F,X}$ is an equivalence for any strong coarse homology theory.  But for the  coarse topological $K$-theory
$K\cX_{\Hilb_{c}(\C)}$ the coarse assembly map is an equivalence for
a much bigger class of bornological coarse spaces, e.g. discrete metric spaces of bounded geometry which admit a coarse embedding into a Hilbert space  \cite[Thm. 1.1]{yu_embedding_Hilbert_space}.
We therefore expect  that the Assumption \ref{werigosetrgertgerwgw}.\ref{ijtgoertgegertgwee1} is satisfied for many spaces not having finite asymptotic dimension.

%

%
  

We consider this note as a possibility to demonstrate the usage of results of coarse homotopy homotopy as developed in \cite{buen}, \cite{ass}, \cite{equicoarse}, \cite{coarsek} and \cite{bel-paschke}.

{\em Acknowledgements: The author thanks M. Ludewig
for fruitful discussion, in particular for insisting to give a detailed argument that \eqref{qwdqwewqrwdqweddwdwdwd1eee} is equivalent to the classical PV-sequence.
This work was supported by the CRC 1085 {\em Higher structures} funded by the DFG.}

\section{The coarse PV-square}
 \label{wkgopgsgfdg}

A $\Z$-equivariant $\bM$-valued coarse homology theory is a functor
$$E^{\Z}:\Z\BC\to \bM$$ from the category $\Z\BC$ of $\Z$-bornological coarse spaces to a cocomplete stable $\infty$-category $\bM$ such that the functor $E^{\Z}$   is coarsely invariant, excisive, $u$-continuous and vanishes on flasques \cite[Def. 3.10]{equicoarse}.  There exists a universal $\Z$-equivariant coarse homology theory \cite[Def. 4.9]{equicoarse}
$$\Yo^{s}:\Z\BC\to \Z\Sp\cX\ ,$$ where $ \Z\Sp\cX$ is a presentable stable $\infty$-category called the $\infty$-category of coarse motivic spectra.

The category $\Z\BC$ admits a symmtric monoidal structure $\otimes$ described in \cite[Ex. 2.17]{equicoarse}. It induces a presentably symmetric monoidal structure on $\Z\Sp\cX$  such that functor $\Yo^{s}$
refines essentially uniquely to a symmetric monoidal functor \cite[Sec. 4.3]{equicoarse}. 

If $E^{\Z}:\Z\BC\to \bM$ is a $\Z$-equivariant coarse homology theory, then by \cite[Cor. 4.10]{equicoarse} it essentially
uniquely factorizes as the composition of  $\Yo^{s}$ and a colimit-preserving functor  $E^{\Z}:  \Z\Sp\cX\to \bM$  denoted by the same symbol.

If $X$ is an object of $\Z\BC$,  then  the functor 
$E^{\Z}(X\otimes -):\Z\BC\to \bM$ is again a $\Z$-equivariant coarse homology theory
called the twist of $E^{\Z}$ by $X$.

We will encounter additional conditions and structures on $E$:
\begin{enumerate}
\item continuity \cite[Def. 5.19]{equicoarse}
\item strongness \cite[Def. 4.19]{equicoarse}
\item strong additivity \cite[Def. 3.12]{equicoarse}
\item transfers \cite[Def. 2.53]{coarsetrans}\ .
\end{enumerate}

For any $\Z$-set $Y$ we can form the objects $Y_{min,max}$ and $Y_{min,min}$ in $\Z\BC$ obtained by equipping $Y$ with the minimal coarse structure (whose maximal entourage is $\diag(Y)$) and the minimal bornology (consisting of the finite subsets) or the maximal bornology (consisting of all subsets), respectively.
The group $\Z$ has a canonical $\Z$-coarse structure generated by the entourages
$U_{r}:=\{(n,m)\in \Z\times \Z\mid |n-m|\le r\}$ for all $r$ in $\nat$. We let $\Z_{can,min}$ in $\Z\BC$ denote the corresponding object.

Let $E^{\Z}:\Z\BC\to \bM$ be an  equivariant coarse homology theory
with transfers. 
We start with 
 describing a  commutative  square 
  \begin{equation}\label{qwdqwewqwdqwed1eee} \xymatrix{ E^{\Z}( -\otimes \Z_{min,min})\ar[r]^-{\iota}\ar[d]& E^{  \Z}(-\otimes \Z_{can,min})\ar[d]^{\tr}\\0\ar[r]& E^{\Z}(-\otimes \Z_{can,min}\otimes \Z_{min,min})}
\end{equation} 
of   $\bM$-valued $\Z$-equivariant coarse homology theories.
 The map $\iota$ is induced by the  morphism $\Z_{min,min}\to \Z_{can,min}$ in $\Z\BC$
given by the identity of the underlying sets.
The map
  $\tr$   is the transfer along the  coarse covering $ \Z_{min,min}\to *$.   \begin{lem}\label{wigowegreewff}
The square \eqref{qwdqwewqwdqwed1eee} commutes.
\end{lem}
\begin{proof}
We have the following commutative diagram
\begin{equation}\label{qewfqwefdewqdewdqwerereeredewqdeeeee}
\xymatrix{ E^{  \Z}(-\otimes \Z_{min,min})\ar[r]^{\iota}\ar[d]^{\tr}&E^{  \Z}(-\otimes \Z_{can,min})\ar[d]^{\tr}\\ E^{\Z}( -\otimes \Z_{min,min}\otimes \Z_{min,min}) \ar[r] \ar[d]_{!}^{\simeq}&E^{\Z}( -\otimes \Z_{can,min}\otimes \Z_{min,min})\ar[d]_{!}^{\simeq}\\ E^{\Z}( -\otimes \Res^{\Z}(\Z_{min,min})\otimes \Z_{min,min})  \ar[r] \ar[d]_{\partial^{MV}}^{0}&E^{\Z}( -\otimes \Res^{\Z}(\Z_{can,min})\otimes \Z_{min,min})\ar[d]_{\partial^{MV}}^{\simeq}\\ \Sigma E^{  \Z}  ( - \otimes \Z_{min,min})\ar@{=}[r]  & \Sigma E^{  \Z}  ( - \otimes \Z_{min,min})}\ .
\end{equation} 
The upper three horizontal  maps are all induced from $\Z_{min,min}\to \Z_{can,min}$.
The map marked by $!$ is induced by the equivariant map of $\Z$-sets
$$X\times \Z\times \Z\to X\times \Res^{\Z}(\Z)\times \Z\ , \quad (x,n,m)\mapsto (x,n-m,m)\ .$$ The morphism  $\partial^{MV}$ is the Mayer-Vietoris boundary  map associated to the decomposition of $\Res^{\Z}(\Z)$ into $(\nat,-\nat)$. 
Its left instance vanishes since this decomposition of $\Res^{\Z}(\Z_{min,min})$ is coarsely disjoint. 
Using the Mayer-Vietoris sequence we see that  right instance of $\partial^{MV}$ is an equivalence
since the subsets $\pm \nat$ of $\Res^{\Z}(\Z_{can,min})$ are flasque.   
The commutativity of the upper square encodes the naturality of the transfer, and the commutativity of the lower square   reflects the  naturality of the  Mayer-Vietoris boundary. 
The middle square arrises from applying $E^{\Z}$ to a commutative square in $\Z\BC$ and therefore commutes, too.

 The commutative  diagram \eqref{qewfqwefdewqdewdqwerereeredewqdeeeee}
yields a filler of  \eqref{qwdqwewqwdqwed1eee}.   \end{proof}

We now fix $X$ in $\Z\BC$ and consider the $\Z$-equivariant coarse homology theory   \begin{equation}\label{qw3erqw3asded}
E_{X}^{\Z}(-):=E^{\Z}(X\otimes -):\Z\BC\to \bM
\end{equation} obtained from $E^{\Z}$ by twisting with $X$.
\begin{ddd}
The commutative square \begin{equation}\label{qwdqwewqwdqwedefe1eee}
\xymatrix{E_{X}^{Z}(\Z_{min,min})
\ar[r]^{\iota}\ar[d]&E_{X}^{\Z}(\Z_{can,min})\ar[d]^{\tr}\\ 0\ar[r]&E_{X}^{\Z}(\Z_{can,min}\otimes \Z_{min,min})}
\end{equation}
is called the coarse PV-square associated to $E^{\Z}$ and $X$.
\end{ddd}

\section{The coarse PV-sequence}
In this section we are interested in conditions on $E^{\Z}$ and $X$ ensuring that the coarse PV-square \eqref{qwdqwewqwdqwedefe1eee} is cartesian, i.e. that \begin{equation}\label{weqfoijoiwqjdoieewdqwedqewd}
E_{X}^{Z}(\Z_{min,min})\stackrel{\iota}{\to}E_{X}^{\Z}(\Z_{can,min})\stackrel{\tr}{\to} E_{X}^{\Z}(\Z_{can,min}\otimes \Z_{min,min})
\end{equation}   is a part of 
a fibre sequence. In this case it will be called the coarse PV-sequence. 

Let 
$i: \Z\BC_{min}\to \Z\BC$  denote the inclusion of the full subcategory of $\Z$-bornological coarse spaces with the minimal bornology.  We let $i^{*}$ and $i_{!}$ denote the  operations of restriction and  left Kan extension along $i$. Recall from \cite[Sec. 5.4]{equicoarse} that a coarse homology theory $F^{\Z}: 
\Z\BC\to \Sp$ is continuous if the canonical transformation \begin{equation}\label{afdasdfqwefq}
 i_{!}i^{*}F^{\Z}\to F^{\Z}
\end{equation}  is an equivalence.  
\begin{ddd}\label{qrigfjqorfqew}We call $ F^{\Z}_{c}:=i_{!}i^{*}F^{\Z}$ the continuous approximation of $F^{\Z}$.\end{ddd}
We apply this construction to the functor  $E_{X}^{\Z}$  from \eqref{qw3erqw3asded} and get a continuous equivariant
coarse homology theory $E_{X,c}^{\Z}$.

  The functor $E_{X,c}^{\Z}$  gives rise to a functor
$$HE^{\Z}_{X,c}:\Z\Orb\to \bM\ , \quad S\mapsto E^{\Z}_{X,c}(S_{min,max}\otimes \Z_{can,min})$$
from the orbit category of $\Z$ to $\bM$ and an associated Davis-L\"uck assembly map  $\mu^{DL}_{HE^{\Z}_{X,c}}$ given in   \eqref{wefqwedqwdeewdeqwd}.  Unfolding the definition,  
the Davis-L\"uck assembly map is   equivalent to the map
\begin{equation}\label{arfefewfqfewqdqweded}
\colim_{B\Z'} E_{X,c}^{\Z}(  \Z_{min,max}\otimes \Z_{can,min}) \to 
E^{\Z}_{X,c}( \Z_{can,min})
\end{equation} induced by the projections $ \Z_{min,max}\to *$, where $\Z'$ acts on $Z_{min,max}$ by translations.

%
 
The main theorem of the present section is:
\begin{theorem}\label{weijtgowegferwerwgwergre}
 Assume:
\begin{enumerate}
\item $X$  has the minimal bornology and is discrete.
\item $E^{\Z}$   is strong and strongly additive.
\end{enumerate}
 Then the PV-square \eqref{qwdqwewqwdqwedefe1eee}   is cartesian if and only if $\mu_{HE^{\Z}_{X,c}}^{DL}$ is an equivalence.
\end{theorem}
\begin{proof}
 We let $\Z'$ be a second copy of $\Z$ which acts on $\Z_{min,min}$ by translations. Then $\Z'$ acts by functoriality on the coarse homology  theory $E^{\Z}_{X}(  -\otimes \Z_{min,min})$. The coarse covering
$\Z_{min,min}\to *$ in $\Z'$-equivariant, where $\Z'$ acts trivially on $*$. As a consequence, the transfer map $\tr$ for $E^{\Z}_{X}$ along the coarse coverings 
$  -\otimes \Z_{min,min}\to  -$ has a factorization
$$\tr:E_{X}^{\Z}(-)\stackrel{\coass_{X}}{\to} \lim_{B\Z'}E_{X}^{\Z}(-\otimes \Z_{min,min})\stackrel{\ev}{\to} E_{X}^{\Z}(-\otimes \Z_{min,min})\ .$$

\begin{ddd} For $Z$ in $\Z\BC$
we call $$\coass_{X,Z}:E^{\Z}_{X}(Z)\to  \lim_{B\Z'}E^{\Z}_{X}(Z\otimes \Z_{min,min})$$ the coassembly map for the object $Z$.
\end{ddd}

In a stable $\infty$-category like $\bM$  finite limits commute with colimits.  Since
$\lim_{B\Z'}$ is a finite limit 
the  functor
$\lim_{B\Z'}E^{\Z}(- \otimes \Z_{min,min})$
is   again a $\Z$-equivariant coarse homology theory. 
We will keep $X$ in the notation since later we will study properties of the coassembly map which may depend on the choice of $X$.
 
\begin{prop}\label{weriojtgowtgerf}
Assume that  $E^{\Z}$ is strong and the coassembly maps   $\coass_{X,\Z_{min,min}}$ is an equivalence. Then the following assertions are equivalent: 
\begin{enumerate} 
\item  The coassembly map    $\coass_{X,\Z_{can,min}}$  is an equivalence.
\item The coarse $PV$-square    \eqref{qwdqwewqwdqwedefe1eee}  is cartesian.
\end{enumerate}
\end{prop}
\begin{proof}
We consider the commutative diagram
\begin{equation}\label{qewfqwefdewqdewdrerqwerereeredewqdeeeee}
\xymatrix{ E_{X}^{  \Z}(  \Z_{min,min})\ar[r]^{\iota}\ar[d]_{\simeq}^{\coass_{X,\Z_{min,min}}}&E_{X}^{  \Z}(  \Z_{can,min})\ar[d]^{\coass_{X,Z_{can,min}}}\ar@/^3.5cm/[dd]^{\tr}\\ \lim_{B\Z'} E_{X}^{\Z}(  \Z_{min,min}\otimes \Z_{min,min}) \ar[r]^{\lim_{B\Z'}\iota} \ar@{..>}[dr]^{0} &\lim_{B\Z'}E_{X}^{\Z}(   \Z_{can,min}\otimes \Z_{min,min})\ar[d]^{\ev} \\     & E_{X}^{\Z}(   \Z_{can,min}\otimes \Z_{min,min}) }\ .
\end{equation} 
The vanishing of the composition $\ev\circ \lim_{B\Z'}\iota$ is a consequence of Lemma \ref{wigowegreewff}.  We claim that  $\ev$ presents 
the cofibre of  $ \lim_{B\Z'}\iota$ 

For the moment left us assume the claim. 
If $\coass_{X,Z_{can,min}}$ is an equivalence, then $\tr$ represents the cofibre of $\iota$
and  the coarse PV-square    \eqref{qwdqwewqwdqwedefe1eee}  is cartesian.
Vice versa, if the coarse PV-square    \eqref{qwdqwewqwdqwedefe1eee}  is cartesian,  
 $\coass_{X,Z_{can,min}}$ is an equivalence by an application of the Five Lemma.

In order to show the claim  we form the commutative diagram  \begin{equation}\label{ewqfqwefwedqwedqwdewd}
\xymatrix{\lim_{B\Z'}E^{\Z}_{X}(\Z_{min,min}\otimes \Z_{min,min})\ar[r]^{\ev}\ar[d]^{\lim_{B\Z'}\iota}& E^{\Z}_{X}(\Z_{min,min}\otimes \Z_{min,min})\ar[d]_{0}^{\iota}\ar[r]^{1-\sigma}&E^{\Z}_{X}(\Z_{min,min}\otimes \Z_{min,min})\ar[d]^{\iota}\\\lim_{B\Z'}E^{\Z}_{X}(\Z_{can,min}\otimes \Z_{min,min})\ar[r]^{\ev}\ar[d]& E^{\Z}_{X}(\Z_{can,min}\otimes \Z_{min,min})\ar[d]\ar[r]_{0}^{1-\sigma}&E^{\Z}_{X}(\Z_{can,min}\otimes \Z_{min,min})\ar[d]\\\lim_{B\Z'} Q\ar[r]\ar[d]^{0}&  Q\ar[r]^{1-\sigma}&  Q\\&&}
\end{equation}
in   $\bM$.
Here $Q$ is the object of $\bM$  defined as the cofibre of the middle map denoted by $\iota$    with the induced action of $\Z'$. The symbol $\sigma$ stands for the action of the generator of $\Z'$. 
The horizontal sequences are fibre sequences  reflecting the usual presentation of the  fixed points of a $\Z'$-object in a stable $\infty$-category.
The vertical sequences are fibre sequences by construction.
We now have the following assertions:
\begin{lem}\label{tgot0grefwefefewf}\mbox{}\begin{enumerate}
\item \label{sfdgsfgdsfds} The map $E^{\Z}_{X}(\Z_{min,min}\otimes \Z_{min,min})\to E^{\Z}_{X}(\Z_{can,min}\otimes \Z_{min,min})$ vanishes.
\item \label{sfdgsfgdsfds1}The group $\Z'$ acts trivially on $E^{\Z}_{X}(\Z_{can,min}\otimes \Z_{min,min})$.
\item\label{sfdgsfgdsfds2} The map $\lim_{B\Z'}\iota$ is split injective.
 \end{enumerate}\end{lem}
These facts imply that  the  maps  marked by $0$  in the diagram \eqref{ewqfqwefwedqwedqwdewd}  vanish.
 From Lemma \ref{afafadsadsf} below applied to the web \eqref{ewqfqwefwedqwedqwdewd}  
 we then get a tcommutative triangle
 \begin{equation}\label{}
\xymatrix{&\lim_{B\Z'}E^{\Z}_{X}(\Z_{can,min}\otimes \Z_{min,min})\ar[dr]^{\ev}\ar[dl]&\\\lim_{B\Z'}Q\ar[rr]^{\simeq}&&Q}
\end{equation} 
 which solves our task.
 
 We consider a web of fibre sequences
\begin{equation}\label{regwergefefefewe}
\xymatrix{A\ar[r]\ar[d]^{a}&B\ar[d]^{0}\ar[r]&C\ar[d]\\E\ar[d]^{l} \ar[r]^{j}&F
\ar[d]^{i}\ar[r]^{0}&G\ar[d]\\H
\ar[r]^{m} \ar[d]^{0}\ar[d]&I\ar[r]^{d}\ar[d]&J\ar[d]\\\Sigma A\ar[r]&\Sigma B\ar[r]&\Sigma C}
\end{equation} 
in some stable $\infty$-category where the maps marked by $0$ vanish.

\begin{lem}\label{afafadsadsf}
There exists a commutative triangle
 \begin{equation}\label{}
 \xymatrix{&E\ar[dr]^{j}\ar[dl]_{l}&\\ H\ar[rr]^{u}_{\simeq}&&F}
\end{equation} 
\end{lem}
\begin{proof}
The map $u$ is obtained from the universal property of $H$   as the cofibre of $a$ together with the fact that $ j\circ a\simeq 0$ witnessed be the upper left square. 
We have $j\simeq u\circ l$.

The potential inverse  $v:F\to H$ of $u$ is obtained from the universal property of $H$ as  the  fibre of $I\to J$ together
with the fact that  $d\circ i\simeq 0$  witnessed by the middle right square. 
We have $i\simeq m\circ v$.

 We claim that $u$ and $v$ are inverse to each other equivalences in the homotopy category of $\bM$. 
 To this end we must    consider equivalences between maps (we say homotopies)
 as $2$-categorical  data.
We first observe that $i\circ u= m$.   The map
 $u$ is  homotopy class of the pair of the map $j$ together with the homotopy 
 $\alpha:j\circ a\Rightarrow 0$.
 Similarly the map $m$ is the homotopy class of the map $i\circ j$ together with the homotopy $\beta:i\circ j\circ a\Rightarrow 0$. The commutativity of the
 left middle square expresses the fact that $i_{*}\alpha=\beta$.
We conclude that $i\circ u= m$. 

  We now calculate
 $i\circ u\circ v\circ  j= m\circ v\circ j= i\circ j
  $. Since $i$ is a  monomorphism and  $j$ is an epimorphism
  we conclude that $u\circ v=\id_{F}$.

We now show that $v\circ u= \id_{H}$.   The map
 $v$ is the homotopy class of a pair of the map $i$ and the zero homotopy $\sigma$ of $d\circ i$. The map $l$ is similarly a homotopy class of the pair  of the map $i\circ j$ and the zero homotopy $\kappa$ of $d\circ i\circ j$.
 In this picture the commutativity of the left middle square expresses the fact that
 $j^{*}\sigma=\kappa$. Alltogether we conclude that
 $j^{*}v=l$. We now calculate that 
   $v \circ u\circ l= v\circ j=  l$. Since $l$ is  is an epimorphism we conclude that  
   $v\circ u=\id_{H}$.\end{proof}

\begin{proof}[Proof of Lemma \ref{tgot0grefwefefewf}]
Assertion \ref{sfdgsfgdsfds}  is provided by the two lower squares of \eqref{qewfqwefdewqdewdqwerereeredewqdeeeee}.

We now show Assertion \ref{sfdgsfgdsfds1}.
 We use the isomorphism 
\begin{equation}\label{rwewfwefewfrwf}
  \Z_{can,min}\otimes \Z_{min,min}\to   \Res^{\Z}(\Z_{can,min})\otimes \Z_{min,min}\ , \quad (m,n)\mapsto (m-n,n)
\end{equation} 
in $\Z \BC$. The $\Z'$-action on  $   \Res^{\Z}(\Z_{can,min})\otimes \Z_{min,min}$ is   given by
$(k,(m,n))\mapsto (m-k,n+k)$. For given $k$  in $\Z'$ we can decompose this into the map
$c_{k}:(m,n)\mapsto (m-k,n)$ and $a_{k}:(m,n)\mapsto (m,n+k)$.
The map $c_{k}$ is close to the identity. By the coarse invariance of  coarse homology theories the  induced action of $k$ on  
$E^{\Z}_{X}(  \Res^{\Z}(\Z_{can,min})\otimes \Z_{min,min})$ is therefore the equivalent map induced by $a_{k}$.  The morphism 
$$\xymatrix{\Res^{\Z}(\Z_{can,min})\otimes \Z_{min,min}\ar[rr]^{a_{k}}\ar[dr]&&\Res^{\Z}(\Z_{can,min})\otimes \Z_{min,min}\ar[dl]\\&\Res^{\Z}(\Z_{can,min})&}$$
between coarse coverings
induces 
 the  commutative diagram
\begin{equation}\label{} 
\xymatrix{&E^{\Z}_{X}(  \Res^{\Z}(\Z_{can,min} ))\ar[dr]^{\simeq}_{\tr}\ar[dl]_{\simeq}^{\tr}&\\  E^{\Z }_{X}( \Res^{\Z}(\Z_{can,min})\otimes \Z_{min,min})\ar[rr]^{a_{k}}&&E^{\Z }_{X}( \Res^{\Z}(\Z_{can,min})\otimes \Z_{min,min})}
\end{equation}
 This shows that  the map induced by  $a_{k}$ is also  equivalent to the identity.

 We finally   show Assertion \ref{sfdgsfgdsfds2}. We use the geometric cone functor
 $\cO^{\infty}:\Z\UBC\to \Z\BC$ which sends a $\Z$-uniform bornological coarse space $Y$ to the $\Z$-bornological coarse space given by
 $\Z$-set $\R\times Y$ with the bornology generated by the subsets
 $[-n,n]\times B$ for all bounded subsets $B$ of $Y$ and $n$ in $\nat$, and the hybrid coarse structure (this is $\cO(Y)_{-}$ in the notation from \cite[Sec. 9]{equicoarse}).
 By \cite[Prop. 9.31]{equicoarse} we have a natural cone fibre sequence in $\Z\Sp\cX$ involving the cone boundary $\partial^{\cone}:\Yo^{s}(\cO^{\infty}(-))\to \Sigma \Yo^{s}(\cF(-))$ which is induced from the transformation 
 $d^{\cone}:\cO^{\infty}(-)\to \R\otimes \cF(-)$  between $\Z\BC$-valued functors  and the suspension equivalence $\Yo^{s}(\R\otimes-)\simeq \Sigma^{s} \Yo^{s}(-)$, where $\cF:\Z\UBC\to \Z\BC$ is the forgetful functor.
 
 We consider $\R$ as an object in $\Z\UBC$ with the action of $\Z$ by translations and the metric structures. We furthermore consider the invariant  subset $\Z$ of $\R$ with the induced structures. Note that $\cF(\Z)\cong \Z_{can,min}$ and that the inclusion $\Z_{can,min}\to \cF(\R)$ is a coarse equivalence. We let $\Z_{\disc}$ in $ \Z\UBC$  denote $\Z$ with the coarse structure replaced by the discrete one. Then the identity map of underlying sets  is a coarsification morphism $\Z_{\disc}\to \Z$ in $\Z\UBC$. We have  $\cF(\Z_{\disc})\cong \Z_{min,min}$, and by \cite[Prop. 9.33]{equicoarse}
 the induced map 
 $\Yo^{s}(\cO^{\infty}(\Z_{\disc}))\to\Yo^{s} (\cO^{\infty}(\Z))$ is an equivalence.
 We consider the following diagram:
 \begin{equation}\label{}
 \xymatrix{E^{\Z}_{X}(\cO^{\infty}(\Z_{\disc}))\ar[r]^{i}\ar[d]^{\coass_{X,\cO^{\infty}(\Z_{\disc})}}&E^{\Z}_{X}(\cO^{\infty}(\R))\ar[d]^{\coass_{X,\cO^{\infty}(\R)}}\\ 
\lim_{B\Z'} E^{\Z}_{X}(\cO^{\infty}(\Z_{\disc})\otimes \Z_{min,min})\ar[r]\ar[d]^{\lim_{B\Z'}\partial^{\cone}_{\Z_{\disc}}}&\lim_{B\Z'} E^{\Z}_{X}(\cO^{\infty}(\R)\otimes \Z_{min,min})\ar[d]^{\lim_{B\Z'} \partial^{\cone}_{\R}}\\
\lim_{B\Z'} \Sigma    E^{\Z}_{X}(\Z_{min,min}\otimes \Z_{min,min})  \ar[r]^{\lim_{B\Z'}\iota}&
\lim_{B\Z'} \Sigma    E^{\Z}_{X}(\Z_{can,min}\otimes \Z_{min,min})
}\ .
\end{equation}
The upper two horizontal maps are induced by $\Z_{\disc}\to \Z\to \R$.
At the lower right corner we used the identification 
$  E^{\Z}_{X}(\Z_{can,min}\otimes \Z_{min,min})\stackrel{\simeq}{\to}
  E^{\Z}_{X}(\cF(\R)\otimes \Z_{min,min})$
  induced by the coarse equivalence   $\Z_{can,min}\to \cF(\R)$.
  \end{proof}

It is clear that the following assertions imply Assertion \ref{tgot0grefwefefewf}.\ref{sfdgsfgdsfds2}. 
\begin{lem} \label{qirogfqrfewfewfqwef}
\mbox{}
\begin{enumerate}
\item \label{wtihgowgergwefw}$\coass_{X,\cO^{\infty}(\Z_{\disc})}$ and $\coass_{X,\cO^{\infty}(\R)}$
are equivalences.
\item \label{wtihgowgergwefw1} $\lim_{B\Z'}\partial^{\cone}_{\Z_{\disc}}$ and $\lim_{B\Z'}\partial^{\cone}_{\R}$ are equivalences.
\item\label{wtihgowgergwefw2} The map $i$ is split injective.
\end{enumerate}
\end{lem}
 \begin{proof}
 We start with Assertion \ref{wtihgowgergwefw}.
 By \cite[Prop. 9.35]{equicoarse} we have an equivalence 
 $\partial^{\cone}_{\Z_{\disc}}:\Yo^{s}(\cO^{\infty}(\Z_{\disc}))\simeq  \Sigma \Yo^{s}(\Z_{min,min})$.
 Hence
 $\coass_{X,\cO^{\infty}(\Z_{\disc})}$ is equivalent to the suspension of 
 $\coass_{X, \Z_{min,min}}$ which is an equivalence by assumption.

 Since the coassembly map  $\coass_{X,-}$ is a transformation between $\Z$-equivariant coarse homology theories
 the  fact that $\coass_{X,Z}$ is an equivalence only depends on the motive $\Yo^{s}(Z)$ in $\Z\Sp\cX$.
 In order to deal with $\cO^{\infty}(\R)$ we use that
 the composition
 $ \Yo^{s}\circ \cO^{\infty}:\Z\UBC\to \Z\Sp\cX$  is excisive
 and homotopy invariant by \cite[Prop. 9.36 \& 9.38]{equicoarse}. In particular, if we decompose
 $\R$   into the $\Z$-invariant subsets $\bigcup_{n\in \nat}[n,1/2+n]$ and
 $\bigcup_{n\in \nat}[n+1/2,n +1]$ and use the 
 homotopy equivalences $\Z\to \bigcup_{n\in \nat}[n,1/2+n]$, $n\mapsto n$
 and
 $\Z\to \bigcup_{n\in \nat}[n+1/2,n +1]$, $n\mapsto n+1$ in $\Z\UBC$,
 we get a Mayer-Vietoris sequence \begin{equation}\label{qwefqwedqewdqdqwed}
 \Yo^{s} ( \cO^{\infty}(\Z))\oplus \Yo^{s}( \cO^{\infty}(\Z) )  \to \Yo^{s}(\cO^{\infty}(\Z))\oplus \Yo^{s} ( \cO^{\infty}(\Z)) \to  \Yo^{s} (\cO^{\infty}(\R))\ .
\end{equation} 
Since $\coass_{X,\cO^{\infty}(\Z)}$ is already known to be an equivalence we  can use the Five Lemma  in order to conclude that $\coass_{X, \cO^{\infty}(\R)}$
  is an equivalence, too.
 
 We now show Assertion \ref{wtihgowgergwefw1}.
 It suffices to show that  the underlying maps $$\partial^{\cone}_{\Z_{\disc}}:
 E^{\Z}_{X}(\cO^{\infty}(\Z_{\disc})\otimes \Z_{min,min})\to \Sigma E^{\Z}_{X}(\Z_{min,min}\otimes Z_{min,min}) 
 $$ and $$\partial^{\cone}_{\R}: 
 E^{\Z}_{X}(\cO^{\infty}(\R)\otimes \Z_{min,min})\to \Sigma E^{\Z}_{X}(\cF(\R)\otimes Z_{min,min})$$ are equivalences.
 As already observed in the previous step, $\partial^{\cone}_{\Z_{\disc}}$ is an equivalence. We now discuss the case of $\partial^{\cone}_{\R}$. 
  We use the isomorphisms
 $$\cO^{\infty}(\R)\otimes \Z_{min,min}\to  \cO^{\infty}( \Res^{\Z}(\R))\otimes \Z_{min,min} \ , \quad (t,x,n)\mapsto  t,x-n,n) $$ and
  $$\cF(\R)\otimes  \Z_{min,min}\to \Res^{\Z}(\cF(\R))\otimes \Z_{min,min}
\ , \quad (x,n)\mapsto (x-n,n)$$
 in $\Z\BC$. 
 It therefore suffices to show that 
 $$\partial^{\cone}_{\Res^{\Z}(\R)}:
 E^{\Z}_{X}(  \cO^{\infty}( \Res^{\Z}(\R))\otimes \Z_{min,min})\to \Sigma E^{\Z}_{X}(\cF(\Res^{\Z}(\R))\otimes Z_{min,min})$$
 is an equivalence. By \cite[Prop. 7.12]{ass} the   uniform bornological  coarse space $\Res^{\Z}(\R)$ is coarsifying.  Since $E^{\Z}$ is assumed to be strong also $E^{\Z}_{X}(-\otimes \Z_{min,min})$ is a strong non-equivariant coarse homology theory. We can now conclude that 
 $\partial^{\cone}_{\Res^{\Z}(\R)}$  is equivalent to coarse assembly map $\mu_{E^{\Z}_{X}(-\otimes \Z_{min,min}),\cF(\Res^{\Z}(\R))}$ 
   from \cite[Def. 9.7]{ass}. Since $\cF(\Res^{\Z}(\R))$ (i.e. the bornological coarse space $\R$ with the metric structures) has weakly finite asymptotic dimension  the coarse assembly map $\mu_{E^{\Z}_{X}(-\otimes \Z_{min,min}),\cF(\Res^{\Z}(\R))}$ is an equivalence by \cite[Thm. 10.4]{ass}.

 We finally show Assertion \ref{wtihgowgergwefw2}.
 We consider the Mayer-Vietoris sequence \eqref{qwefqwedqewdqdqwed}. 
Using the intersection  with $\Z$ of   decomposition of $\R$ into the subsets  $\bigcup_{n\in \nat}[n,1/2+n]$ and
 $\bigcup_{n\in \nat}[n+1/2,n +1]$   we get an analogous
Mayer-Vietoris sequence for $\Yo^{s} (\cO^{\infty}(\Z_{\disc}))$.
The inclusion $\Z_{\disc}\to \R$  induces the map of Meyer-Vietoris sequences (since $\cO^{\infty}$ is invariant under coarsification we can omit the subscript $\disc$)
 $$\xymatrix{ \Yo^{s} ( \cO^{\infty}(\Z ) \ar[r]^{x\mapsto x\oplus \sigma (x)}\ar[d]^{x\mapsto x\oplus 0}&\ar@{-->}@/^-1cm/[l]_{x\oplus y\mapsto x}  \Yo^{s}(\cO^{\infty}(\Z))\oplus \Yo^{s} ( \cO^{\infty}(\Z))\ar@{=}[d]\ar[r]^-{e}& \Yo^{s} (\cO^{\infty}(\Z ))\ar[d]^{\alpha}\\\ar@/^1cm/@{..>}[u]^{(a,b)\mapsto a+b} \Yo^{s}(\cO^{\infty}(\Z))\oplus \Yo^{s} ( \cO^{\infty}(\Z))\ar[r]^{x\oplus y\mapsto x+y \oplus \sigma(x)+\sigma(y)}&   \Yo^{s}(\cO^{\infty}(\Z))\oplus \Yo^{s} ( \cO^{\infty}(\Z))\ar[r]& \Yo^{s} (\cO^{\infty}(\R))\ar@/^-1cm/@{..>}[u]_{\beta}}\ ,$$
 where $\sigma$ indicates the map induced by action of the generator of $\Z$.
 The existence of the split indicated by the dashed arrow implies that $e$ is 
 an epimorphism.
The left vertical arrow has a left-inverse (indicated by the left dotted arrow) such that corresponding left square commutes. It induces a map $\beta$ as indicated.   We have $\beta\circ \alpha \circ e\simeq e$ and hence $\beta\circ \alpha\simeq \id_{\Yo^{s} ( \cO^{\infty}(\Z))}$. Applying $E^{\Z}_{X}$ we obtain the desired  left-inverse $E^{\Z}_{X}(\beta)$ of $i\simeq E^{\Z}_{X}(\alpha)$.
 \end{proof}

This completes the proof of Proposition \ref{weriojtgowtgerf}. \end{proof}

 \begin{prop} \label{wtiohgwtrwerrefwre}If $E^{\Z}$ is strongly additive and 
$X$ is discrete, then $\coass_{X,\Z_{min,min}}$ is an equivalence.
\end{prop}
\begin{proof}
We first calculate the target of the coassembly map explicitly.
\begin{eqnarray}
\lim_{B\Z'}E^{\Z}(X\otimes \Z_{min,min}\otimes \Z_{min,min})&\stackrel{\eqref{rwewfwefewfrwf}}{\simeq}
&\lim_{B\Z'}E^{\Z}(X\otimes \Res^{\Z}(\Z_{min,min})\otimes \Z_{min,min})\nonumber\\ 
&\stackrel{!}{\simeq}&\lim_{B\Z'}\prod_{ \Res^{\Z}(\Z )}E^{\Z}(X\otimes   \Z_{min,min})\nonumber\\&\stackrel{\ev_{0}}{\simeq} &
E^{\Z}(X\otimes   \Z_{min,min})\label{wtrko0wergergrefwferf}
\end{eqnarray}
For the equivalence marked by $!$
we use the isomorphism  $$X\otimes \Res(\Z_{min,min})\otimes \Z_{min,min} \cong \bigsqcup^{\free}_{\Res^{\Z}(\Z )} X\otimes \Z_{min,min}$$ (see  \cite[Ex. 2.16]{equicoarse} for the free union) in $\Z\BC$.
At this point it is important that $X$ is discrete.
The equivalence $!$ then follows from the assumption that $E^{\Z}$ is strongly additive. The group $\Z'$ acts freely transitively on the index set $ \Res^{\Z}(\Z )$ by translations, and also on $\Z_{min,min}$. 
The equivalence $\ev_{0}$ is the projection onto the factor with index $0$ in $\Res^{\Z}(\Z )$. Using \cite[Lem. 2.59]{coarsetrans}, or more concretely    \cite[(2.22)]{coarsetrans}, we see that  the  composition of the transfer with  \eqref{wtrko0wergergrefwferf} is equivalent to the identity. 
Hence   \eqref{wtrko0wergergrefwferf} is an inverse equivalence for  $\coass_{X,\Z_{min,min}}$.
\end{proof}

 We recall the fibre sequence
of functors \begin{equation}\label{qewfqwedqewdewedq}
\Sigma^{-1}F^{\infty}\stackrel{\beta}{\to}  F^{0}\to F
\end{equation}
from $\Z\BC$ to $\Z\Sp\cX$
 introduced in \cite[Def. 11.9 \& (11.2)]{equicoarse}, where $\beta$ is called the motivic forget-control map and $F^{0}\simeq \Yo^{s}$. 
The  motivic forget control map induces a the forget control map
 \begin{equation}\label{fwerfwerferwfrefw}
\gamma_{E_{X}^{\Z}}: E_{X}^{\Z}(  F^{\infty}(\Z_{can,min}))\to \Sigma E_{X}^{\Z}(   F^{0}(\Z_{can,min}))\simeq  \Sigma E_{X}^{\Z}(   \Z_{can,min})\ .
\end{equation}

\begin{prop}\label{wreigjwerijogweferfwefr}
Assume:
\begin{enumerate}
\item $X$  has the minimal bornology and is discrete.
\item $E^{\Z}$ is strong  and strongly additive.
\end{enumerate}
 Then $\coass_{X,\Z_{can,min}}$ is an equivalence if and only if $\gamma_{E^{\Z}_{X}}$ is an equivalence.
\end{prop}
\begin{proof}

%
%
%

 
 The following commutative square \begin{equation}\label{adfasdfdssad}
\xymatrix{ \Sigma^{-1}E^{\Z}_{X}(   F^{\infty}(\Z_{can,min}))\ar[r]^{\gamma_{E^{\Z}_{X}}}\ar[d]^{\Sigma^{-1}\coass_{ X,F^{\infty}(\Z_{can,min})}}&E_{X}^{\Z}(  \Z_{can,min})\ar[d]^{\coass_{X, \Z_{can,min}}}\\ \lim_{B\Z'}\Sigma^{-1}E_{X}^{\Z}( F^{\infty}( \Z_{can,min})\otimes \Z_{min,min})
\ar[r]^{\lim_{B\Z'}\delta}& \lim_{B\Z'}\Sigma^{-1}E_{X}^{\Z}(    \Z_{can,min}\otimes \Z_{min,min})}
\end{equation}
 is at the heart of the proof of the split injectivity of the Davis-L\"uck assembly map for CP-functors given in \cite{desc}. In the present paper we reverse the flow of information and use it in order to deduce properties of the coassembly map 
 $\coass_{X, \Z_{can,min}}$.
 The map $\delta$   in \eqref{adfasdfdssad} is also induced by the forget control map $\beta$ in \eqref{qewfqwedqewdewedq}.
  
 Let $P_{U}(\Z_{can,min})$ in $\Z\UBC$  be the Rips complex  of $\Z_{can,min}$ of size $U$, where $U$ is any invariant  entourage of $\Z_{can,min}$.  
 By definition we have  $$F^{\infty}(\Z_{can,min})\simeq \colim_{U\in \cC_{\Z_{can,min}}} \Yo^{s}(\cO^{\infty}(P_{U}(\Z_{can,min})))\ .$$
 If
 $U_{1}=\{(n,m)\mid |n-m|\le 1 \}$, then we have a natural identification
 $\R\cong P_{U_{1}}(\Z_{can,min})$. We then observe that the maps
 $\R\to P_{U_{r}}(\Z_{can,min})$ are homotopy equivalences  in $\Z\UBC$ for all $r\in [1,\infty)$. This implies  the equivalence  $\Yo^{s}(\cO^{\infty}(\R))\stackrel{\simeq}{\to} F^{\infty}(\Z_{can,min})$.
The proof of  Lemma \ref{qirogfqrfewfewfqwef}.\ref{wtihgowgergwefw}  actually shows that   $\coass_{X,\cO^{\infty}(\R)}$ is an equivalence
provided $\coass_{X,\Z_{min,min}}$ is an equivalence. 
But this is the case by Proposition \ref{wtiohgwtrwerrefwre}.
We conclude that $\coass_{X,F^{\infty}(\Z_{can,min})}$ is an equivalence.
 Note that we used here the assumptions that $E^{\Z}$ is strongly additive and that $X$ is discrete. 
 
In order to show that $\lim_{B\Z'}\delta$ is an equivalence it suffices to show that the underlying map $\delta$ is one.   The isomorphism \eqref{rwewfwefewfrwf} induces the vertical equivalences in the commutative square \begin{equation}\label{fqwefqwefwedqdqde}
\xymatrix{ \Sigma^{-1}E_{X}^{\Z}( F^{\infty}( \Z_{can,min})\otimes \Z_{min,min})\ar[r]^{\delta}\ar[d]^{\simeq}& \Sigma^{-1}E_{X}^{\Z}(    \Z_{can,min}\otimes \Z_{min,min})\ar[d]^{\simeq}\\
  \Sigma^{-1}E_{X}^{\Z}( F^{\infty}(\Res^{\Z}( \Z_{can,min}))\otimes \Z_{min,min})
\ar[r]^{\delta'}& \Sigma^{-1}E_{X}^{\Z}(   \Res^{\Z}( \Z_{can,min})\otimes \Z_{min,min})}\ .
\end{equation} 
The map $\delta'$ is the coarse Baum-Connes assembly map  from \cite[Def. 9.7]{ass}
for the non-equivariant coarse homology theory $E^{\Z}_{X}(-\otimes \Z_{min,min})$ which is strong since $E^{\Z}$ was assumed to be strong.
  Since $\Res^{\Z}(\Z_{can,min})$ has finite asymptotic dimension this coarse Baum-Connes assembly map   is an equivalence by \cite[Thm. 10.4]{ass}. We conclude that the morphism $\lim_{B\Z'}\delta$ in  \eqref{adfasdfdssad}
is an equivalence. 

We now have shown that the morphisms  $\lim_{B\Z'}\delta $ and $
\coass_{X,F^{\infty}(\Z_{can,min})} $ in \eqref{adfasdfdssad} are equivalences.
Hence $\coass_{X, \Z_{can,min}}$ is an equivalence if and only of 
$\gamma_{E^{\Z}_{X}}$ is an equivalence.
 \end{proof}

\begin{prop}\label{werijgoerwgrewferfwf}
The Davis-L\"uck assembly map
 $\mu^{DL}_{HE_{X,c}^{\Z}}$ is equivalent to the forget control map 
 $\gamma_{E_{X}^{\Z}}$ in \eqref{fwerfwerferwfrefw}.
 \end{prop}
 \begin{proof}
 Recall that $E^{\Z}_{X,c}$ is the continuous approximation of $E^{\Z}_{X}$, see \eqref{afdasdfqwefq}.
  \begin{lem}\label{weiorgwergerfwf}
The forget control map $\gamma_{E_{X}^{\Z}}$ is  equivalent to  $\gamma_{E_{X,c}^{\Z}}$.
\end{lem}
\begin{proof}
The full category $\cC\cE$ of $\Z\Sp\cX$ of objects $W$ such that 
$E^{\Z}_{X,c}(W)\to  E^{\Z}_{X}(W)$ is an equivalence is localizing. It contains the objects  $\Yo^{s}(Y)$ for all $Y$ in $\Z\BC$ with the minimal bornology, so in particular $\Yo^{s}(\Z_{can,min})$
and $\Yo^{s}(\Z_{min,min})$.
 The Mayer-Vietoris sequence \eqref{qwefqwedqewdqdqwed} and the   equivalence
$\Yo^{s}(\cO^{\infty}(\Z))\simeq \Sigma \Yo^{s}(\Z_{min,min})$ 
 \cite[Prop. 9.33 \& 9.35]{equicoarse}
 imply that
$\Yo^{s}(\cO^{\infty}(\R))$ belongs to $\cC\cE$. Finally, in the proof of Proposition \ref{wreigjwerijogweferfwefr} we have seen that
$F^{\infty}(\Z_{can,min})\simeq \Yo^{s}(\cO^{\infty}(\R))$.
Therefore
$F^{\infty}(\Z_{can,min})$ belongs to $\cC\cE$.
The natural transformation  
$E^{\Z}_{X,c}\to E^{\Z}_{X}$ induces a commutative square 
$$\xymatrix{E^{\Z}_{X,c}(F^{\infty}(\Z_{can,min}))\ar[r]^{\gamma_{E_{X,c}^{\Z}}}\ar[d]&\Sigma E_{X,c}(\Z_{can,min})\ar[d]\\
E^{\Z}_{X}(F^{\infty}(\Z_{can,min}))
\ar[r]^{\gamma_{E_{X}^{\Z}}}&
 \Sigma E_{X}(\Z_{can,min})}\ .$$
 The observations above imply that the vertical morphisms are equivalences.
\end{proof}

    By \cite[Cor. 8.25]{desc}  applied to  the continuous coarse homology theory $E^{\Z}_{X,c}$
 the Davis-L\"uck assembly map
 $\mu^{DL}_{HE_{X,c}^{\Z}}$ is equivalent to 
 the forget control map \begin{equation}\label{rfwreffewrferf}
 E_{X,c}^{\Z}(  F^{\infty}(\Z_{can,min})\otimes \Z_{max,max})\to   \Sigma E_{X,c}^{\Z}(   \Z_{can,min}\otimes \Z_{max,max})
\end{equation}
 induced by the map $\beta$ in \eqref{qewfqwedqewdewedq}. In the following we explain  that
 the additional factor  $\Z_{max,max}$    can be dropped. First of all,  since $\Z$ is torsion free,   the projection 
 $F^{\infty}(\Z_{can,min})\otimes \Z_{max,max}\to F^{\infty}(\Z_{can,min})$ is an equivalence. Furthermore, the projection
  $\Z_{can,min}\otimes \Z_{max,max}\to \Z_{can,min}$ is even a coarse equivalence.
 Thus the map in \eqref{rfwreffewrferf} is equivalent to  $\gamma_{E_{X,c}^{\Z}}$. By Lemma \ref{weiorgwergerfwf} 
 we conclude that
 $\mu^{DL}_{HE_{X,c}^{\Z}}$  is equivalent to $\gamma_{E_{X}^{\Z}}$.
 \end{proof}

The proof of  Theorem  \ref{weijtgowegferwerwgwergre} now follows from a combination of the assertions of Propositions \ref{qewfqwefdewqdewdrerqwerereeredewqdeeeee}, \ref{wtiohgwtrwerrefwre}, \ref{wreigjwerijogweferfwefr}  and \ref{werijgoerwgrewferfwf}
\end{proof}


We call a morphism in $\Z\Sp\cX$  a local equivalence it is sent to an 
equivalence by the coarse homology theories 
$E^{\Z}(-\otimes \Z_{min,min})$, $E^{\Z}(-\otimes \Z_{can,min})$, and
$E^{\Z}(-\otimes \Z_{can,min}\otimes \Z_{min,min} )$ appearing the coarse PV-square. Note that in view of Remark \ref{wrgijowergwregrwef} below we could drop the last entry of this list without changeing the notion of a local equivalence.  \begin{ddd}\label{wgkoowegkorpwfr} We let $$\ell:\Z\Sp\cX\to \Z\Sp\cX_{\loc}$$ be the localization at 
the local equivalences and set $\Yo^{s}_{\loc}:=\ell \circ \Yo^{s}$. \end{ddd}
Note that the notion of a local equivalence and $\ell:\Z\Sp\cX\to \Z\Sp\cX_{\loc}$
depend on the choice of $E^{\Z}$ though this is not indicated in the notation.
 The  three coarse homology theories listed above have colimit-preserving factorizations $\Z \Sp\cX_{\loc}\to \bM$ which will be denoted by the same symbols.  The property that the coarse PV-square \eqref{qwdqwewqwdqwedefe1eee} associated to $E^{\Z}$ and $X$ is cartesian only depends on the class $\Yo^{s}_{\loc}(X)$. 
 \begin{ddd} We let $\PV_{E^{\Z}}$ denote the full subcategory of $ \Z\Sp\cX_{\loc}$
 of objects $X$ for which the coarse PV-square \eqref{qwdqwewqwdqwedefe1eee} is cartesian.
 \end{ddd}
 By construction $\PV_{E^{\Z}}$ is  a localizing subcategory of $\Z\Sp\cX_{\loc}$.

 \begin{ddd} \label{eiojgwergrefwerf}We define
 $\Z\Sp\cX_{\loc}\langle DL \rangle$ as the localizing subcategory 
 generated by $\Yo^{s}_{\loc}(Y_{min,min})$ for all $\Z$-sets $Y$ such that
  $\mu^{DL}_{HE^{\Z}_{Y_{min,min},c}}$ is an equivalence.
 \end{ddd}

Theorem \ref{weijtgowegferwerwgwergre} now has the following immediate consequence.
 \begin{kor}\label{wrthiojwergwergwerrwfg}
 If  $E^{\Z}$  is strong    and strongly additive, then 
   $$\Z\Sp\cX_{\loc}\langle DL \rangle\subseteq  \PV_{E^{\Z}}\ .$$
 \end{kor}
   \begin{prop} If  $E^{\Z}$    is strong and strongly additive, then 
   $\Yo^{s}_{\loc}(\Z_{min,min})\in \Z\Sp\cX_{\loc}\langle  DL \rangle$.
 \end{prop}
 \begin{proof}
 The map $\delta$ in   \eqref{fqwefqwefwedqdqde} is the forget-control map
 $\gamma_{E^{\Z}_{\Z_{min,min}}}$, and it has been shown in the proof  of Lemma \ref{wreigjwerijogweferfwefr} that it is an equivalence.
 We now apply Proposition \ref{werijgoerwgrewferfwf}.
  \end{proof}

%

 \begin{ex}\label{wrthokwegfwerwrf}
 In this example we show that it can happen that $\Yo^{s}_{\loc}(*)\not\in \PV_{E^{\Z}}$. 
 Let $\bA$ be an additive category with strict $\Z$-action and let $E^{\Z}:=K\bA\cX^{\Z}$ denote the coarse algebraic $K$-homology \cite[Def. 8.8]{equicoarse}.
 This functor is strong \cite[Prop. 8.18]{equicoarse}, continuous \cite[Prop. 8.17]{equicoarse}, strongly additive \cite[Prop. 8.19]{equicoarse} and has transfers \cite[Thm. 1.4]{coarsetrans}. 
 We have an obvious equivalence $\mu^{DL}_{HK\bA^{\Z}_{*,c}}\simeq \mu^{DL}_{HK\bA^{\Z}}$.
 The Davis-L\"uck assembly map $\mu^{DL}_{HK\bA^{\Z}}$
 is known to be split-injective (see e.g. \cite{desc} through this is not the
 original reference),   and its cofibre can be expressed in terms of so-called nil-terms, see e.g.  \cite{L_ck_2016}.  
 So $\Yo^{s}_{\loc}(*)\in  \PV_{K\bA\cX^{\Z}}$ if and only if these nil-terms vanish.  
 If $R$ is a non-regular ring and $\bA=\Mod(R)$ with the 
 trivial $\Z$-action, then the nil-terms can be non-trivial and hence
  $\Yo^{s}_{\loc}(*)\not\in  \PV_{K\bA\cX^{\Z}}$.  \hB
 \end{ex}
 
We finally calculate the boundary map of the coarse PV-sequence \eqref{weqfoijoiwqjdoieewdqwedqewd} explicitly. 
Let $X$ be in $\Z\BC$. 
Recall that  $\Z'$ acts  $  E_{X}^{\Z}(  \Z_{min,min})$ by functoriality via its action on $\Z_{min,min}$ by translations.    \begin{prop}\label{werkgjowergwerfwerferfewrfwerf} If  the coassembly maps $\coass_{X, \Z_{min,min}}$ and $\coass_{X,\Z_{can,min}}$ are equivalences, then the coarse $PV$-sequence is equivalent to a fibre sequence
  $$ \Sigma^{-1}E_{X}^{\Z}(\Z_{can,min} )\to  E_{X}^{\Z}(  \Z_{min,min}) \stackrel{1-\sigma}{\to}    E_{X}^{\Z}(  \Z_{min,min})$$
  \end{prop}
\begin{proof}
If $\coass_{X, \Z_{min,min}}$ and $\coass_{X,\Z_{can,min}}$ are equivalences, then in view of Proposition \ref{weriojtgowtgerf}, \eqref{qewfqwefdewqdewdrerqwerereeredewqdeeeee}  and \eqref{ewqfqwefwedqwedqwdewd} the coarse PV-sequence \eqref{weqfoijoiwqjdoieewdqwedqewd}  is equivalent
 to a fibre  sequence
 $$  E_{X}^{\Z}(\Z_{can,min} )\to  E_{X}^{\Z}(\Z_{can,min}\otimes \Z_{min,min}) \stackrel{1-\sigma}{\to} E_{X}^{\Z}(\Z_{can,min}\otimes \Z_{min,min}) \ .$$ The isomorphism \eqref{rwewfwefewfrwf}   yields the  first equivalence in the chain of equivalences
 \begin{equation}\label{gwegwergrefrew}E_{X}^{\Z}(\Z_{can,min}\otimes \Z_{min,min}) \simeq   E_{X}^{\Z}(\Res^{\Z}(\Z_{can,min})\otimes \Z_{min,min}) \simeq  \Sigma E_{X}^{\Z}(  \Z_{min,min}) \end{equation} 
  of $\Z'$-objects in $\bM$. 
  In order to see the second equivalence note that after applying  the isomorphism  \eqref{rwewfwefewfrwf} 
  the group $\Z'$ acts diagonally on $\Res^{\Z}(\Z_{can,min})\otimes \Z_{min,min})$. But arguing as in the proof of Lemma   \ref{tgot0grefwefefewf} we can replace the $\Z'$-action on the factor $\Res^{\Z}(\Z_{can,min})$ by the trivial action.    \end{proof} 
  
If $E^{\Z}$ is strong and  strongly additive, $\Yo^{s}_{\loc}(X)$ is in $\Z\Sp\cX_{\loc}\langle \disc\rangle$, and $\mu_{HE^{\Z}_{X,c}}$ is an equivalence, then  the assumptions of Proposition \ref{werkgjowergwerfwerferfewrfwerf} are satisfied.

 \begin{rem}\label{wrgijowergwregrwef}
 Using the equivalence \eqref{gwegwergrefrew} we observe that a morphism in $\Z\Sp\cX$ is a local equivalence if and only it is sent to an equivalence by $E^{\Z}(-\otimes \Z_{min,min})$ and $E^{\Z}(-\otimes \Z_{can,min})$. \hB
 \end{rem}

 \section{Topological coarse $K$-homology}

 Let $\bC$  be a    $C^{*}$-category with a strict $  \Z$-action     which admits  all  orthogonal AV-sums \cite[Def. 7.1]{cank}.  Then we can consider the spectrum-valued strong coarse homology theory
$$  K\cX^{\Z}_{\bC}:\Z\BC\to \Sp$$
  from \cite[Def. 6.1.2]{coarsek} for the homological functor $\Kcat:\nCcat\to \Sp$. In order to simplify the notation, in the present paper we omit the subscript $c$ appearing in this reference which  indicates continuity.   
 The coarse homology theory   $K\cX^{\Z}_{\bC}$ is  continuous  \cite[Thm. 6.3]{coarsek}, strong \cite[Prop. 6.5]{coarsek}, strongly additive
 \cite[Thm 11.1]{coarsek} and admits transfers by \cite[Thm 9.7]{coarsek}.
 So we can  take $ K\cX^{\Z}_{\bC}$ as an example for $E^{\Z}$ in the preceding sections and define the associated stable $\infty$-category $\Z\Sp\cX_{\loc}$ as in Definition \ref{wgkoowegkorpwfr}. 
 
\begin{ddd}We let $\Z\Sp\cX_{\loc}\langle \disc\rangle$ denote the localizing subcategory of $\Z\Sp\cX_{\loc}$ generated by the objects $\Yo^{s}_{\loc}(Y_{min,min})$ for all $Y$ in $\Z\Set$.  \end{ddd}Recall the Definition \ref{eiojgwergrefwerf} of  $\Z\Sp\cX_{\loc}\langle DL \rangle$.
 \begin{theorem}\label{weitgowergerfrwferfw}
 We have
$\Z\Sp\cX_{\loc}\langle \disc\rangle\subseteq \Z\Sp\cX_{\loc}\langle DL \rangle$.
 \end{theorem}
\begin{proof}
The main input is that the amenable group $\Z$ satisfies the Baum-Connes conjecture with coefficients. 
Using the main result of \cite{kranz}, see also \cite[Thm, 1.7]{bel-paschke}, on the level of homotopy groups  the Davis-L\"uck assembly map 
 $\mu_{HK\cX^{\Z}_{\bC}}$ from \eqref{wefqwedqwdeewdeqwd} is  isomorphic  to the Baum-Connes assembly map \eqref{eqwfwefwdwe} 
   with coefficients in the  free $\Z$-$C^{*}$-algebra $A^{f}(\bC)$ generated by $\bC$. The latter   was introduced in \cite[Def. 3.7]{joachimcat}. Consequently, we know that $\mu^{DL}_{HK\cX_{\bC}^{\Z}}$ is an equivalence. 
   
 Let $Y$ be a $\Z$-set.
 We form the $C^{*}$-category with $\Z$-action $\bV_{\bC}^{  \Z}(Y_{min,min}\otimes \Z_{min,min})$  by specializing \cite[Def. 4.19.2]{coarsek} (note that in the present paper  we use a different notation).  The 
  $\Z$-action is induced by functoriality from the right action on $\Z_{min,min}$. 
  As explained in \cite[Sec. 4]{coarsek} the objects of $ \bV_{\bC}^{  \Z}(Y_{min,min}\otimes \Z_{min,min})$ 
  are triples $(C,\rho,\mu)$, where $(C,\rho)$ is a $G$-object in the multiplier category $ \bM\bC$ of $\bC$  (see \cite[Def. 3.1]{coarsek}), and $\mu$ is an invariant  finitely additive measure on $Y_{min,min}\otimes \Z_{min,min} $ with values in multiplier projections on $C$
  such that $ \mu(\{(y,n)\})\in \bC$ for every point $(y,n)\in Y\times \Z$.
  The morphisms of $\bV_{\bC}^{  \Z}(Y_{min,min}\otimes \Z_{min,min})$  are   the invariant and $\diag(Y\times \Z)$-controlled  multiplier morphisms.  
  
  We need an AV-sum completion  $\bC_{Y}$ of   $\bV_{\bC}^{  \Z}(Y_{min,min}\otimes \Z_{min,min})$,  see \cite[Def. 7.1]{cank} for the notion of an AV-sum. 
  It turns out to be useful to work with  an explicit model for $\bC_{Y}$ given as follows.
The objects of $\bC_{Y}$ are again triples $(C,\rho,\mu)$  as above, but we replace the condition $\mu(\{(y,n)\})\in \bC$ (which prevents the existence of infinite AV-sums in $\bV_{\bC}^{  \Z}(Y_{min,min}\otimes \Z_{min,min})$) by the more general condition  that       the images of  $\mu(\{(y,n)\})$ for  all $(y,n)$ in $Y\times \Z$ are
  isomorphic to  AV-sums of families of  unital objects of $\bC$ (see \cite[Def. 2.14]{cank}). 
Morphisms $A: (C,\rho,\mu)\to (C',\rho',\mu')$   in this category are invariant $\diag(X\times \Z)$-controlled 
  morphisms $A:C\to C'$ in $\bM\bC$ such that $A\mu(\{(y,n)\})\in \bC$ for all $(y,n)$ in $Y\times \Z$. 
 Note that $\bC_{Y}$ contains  $\bV_{\bC}^{  \Z}(Y_{min,min}\otimes \Z_{min,min})$
 as  the full subcategory of unital objects.
  
%
%
%

%
%
%
%


We let $K\cX^{\Z}_{\bC,Y_{min,min},c}$ be the continuous approximation (see Definition \ref{wefqwedqwdeewdeqwd} ) of the coarse homology theory $K\cX^{\Z}_{\bC}(-\otimes Y_{min,min})$.

  \begin{prop}\label{wiothgerththerh}
We have an equivalence
$ K\cX_{\bC,Y_{min,min},c }^{\Z}\simeq K\cX_{\bC_{Y}}^{\Z}$ of  $\Z$-equivariant  coarse homology theories.
\end{prop}
\begin{proof}
We consider the inclusion of groups $\Z\to \Z\times \Z$ given by  $n\mapsto (n,0)$.
We let $\Z\times \Z$ act on $\bC$ via the projection onto the first factor.
Identifying $\Ind_{\Z}^{\Z\times \Z}(-)\simeq - \otimes \Z_{min,min}$
we have the induction equivalence \cite[10.5.1]{coarsek}
  \begin{equation}\label{qefqwedewdqwedqwd}
 K\cX^{\Z}_{\bC,Y_{min,min}}(  -) \stackrel{\simeq}{\to} K\cX^{  \Z\times \Z}_{\bC}(Y_{min,min}\otimes   (-)\otimes  \Z_{min,min})\ .
\end{equation}
On the r.h.s. the first copy of $\Z$ acts diagonally  on  $Y\times \Z\times (-)$, while the second copy only acts on $\Z_{min,min}$. 
We have a natural  isomorphism
$$Y\times  (-)\times \Z\stackrel{\cong}{\to}\ Y\times   \Z\times (-) \ , \quad  (x,- ,n)\mapsto (x,n,n^{-1}-)$$
of functors from $\Z\BC$ to $(\Z\times \Z)\BC$. On the target the first  factor of $\Z\times \Z$ only acts on $ Y\times  \Z$, while the second factor now acts diagonally on $ \Z\times (-)$. We get an equivalence
\begin{equation}\label{eqwfwedwqedewdqewd}
K\cX^{  \Z\times \Z}_{\bC}(Y_{min,min}\otimes   (-)\otimes  \Z_{min,min})\simeq K\cX^{  \Z\times \Z}_{\bC}(Y_{min,min}\otimes     \Z_{min,min}\otimes (-))\ .
\end{equation}  By definition \cite[Def. 6.1]{coarsek} of the coarse $K$-homology functor we have an equivalence
 $$K\cX^{  \Z\times \Z}_{\bC}(Y_{min,min}\otimes     \Z_{min,min}\otimes (-))\simeq \Kcat (\bV_{\bC}^{  \Z\times \Z}(Y_{min,min}\otimes \Z_{min,min}\otimes   (-)))\ .$$
We now construct a natural transformation 
  \begin{equation}\label{dafadfdsfadfafafasdfadfadsf}
\bV_{\bC}^{  \Z\times \Z}(Y_{min,min}\otimes \Z_{min,min}\otimes   Z)\to \bV^{\Z}_{\bC_{Y}}(Z)\ ,
\end{equation}
where 
  $Z$ runs  over  $\Z\BC_{min}$.    Let
     $(C,\rho,\mu)$ be an object $\bV_{\bC}^{  \Z\times \Z}(Y_{min,min}\otimes \Z_{min,min}\otimes   Z)$. The functor \eqref{dafadfdsfadfafafasdfadfadsf}   sends this object to an
 object $((C,\rho_{|\Z\times 1},\pr_{|Y\times \Z,*}\mu),\sigma,\kappa)$.  We first observe that $(C,\rho_{|\Z\times 1},\pr_{|Y\times \Z,*}\mu)$ is an object of $\bC_{Y}$. We define the measure $\kappa$ by  $\kappa(\{z\})=\mu(Y\times \Z\times \{z\})$. Note that $\kappa(\{z\})$ is an endomorphism of
 $(C,\rho_{|\Z\times 1},\pr_{|Y\times \Z,*}\mu)$ in $\bC_{Y}$.
 We further set $\sigma:=\rho_{|1\times \Z}$.
 Using that $Z$ has the minimal bornology we then observe that 
 $((C,\rho_{|\Z\times 1},\pr_{|Y\times \Z,*}\mu),\sigma,\kappa)$ indeed belongs to $\bV^{\Z}_{\bC_{Y}}(Z)$.
  
On morphisms the  functor \eqref{dafadfdsfadfafafasdfadfadsf} is given by the identity.
Naturality in $Z$ is obvious. The transformation
  identifies invariant and controlled morphisms on both sides. 
We conclude that the transformation is fully faithful.

We finally observe that it is essentially surjective. 
Let $((C,\tilde \rho,\tilde \mu),\sigma,\kappa)$ be an object of $\bV^{\Z}_{\bC_{Y}}(Z)$.
We define
$\mu$ such that $ \mu(\{(y,n,z)\}=\tilde \mu(\{(y,n)\})\kappa(z)$ and
$\rho$ such that  by $\rho_{(n,m)}=\tilde \rho_{n}\sigma_{m}$.
Then $(C,\rho,\mu)$ is a preimage in $\bV_{\bC}^{  \Z\times \Z}(Y_{min,min}\otimes \Z_{min,min}\otimes   Z)$.

Applying $\Kcat$ to the natural equivalence \eqref{dafadfdsfadfafafasdfadfadsf}
 to we get a natural  equivalence \begin{equation}\label{adfasq2fasdfdsf}
\Kcat (\bV_{\C}^{  \Z\times \Z}(X_{min,min}\otimes \Z_{min,min}\otimes   Z))\stackrel{\simeq}{\to}    K\cX_{\bC_{Y}}^{\Z}(Z)
\end{equation}
for $Z$ in  $\Z\BC_{min}$.  Composing
\eqref{qefqwedewdqwedqwd}, \eqref{eqwfwedwqedewdqewd} and \eqref{adfasq2fasdfdsf} we get the equivalence \begin{equation}\label{qewfewdqdewdqewd}
K\cX^{\Z}_{\bC,Y_{min,min}} \simeq   K\cX_{\bC_{Y}}^{\Z} 
\end{equation}
of functors on $\Z\BC_{min}$.   The desired equivalence is now given by
$$K\cX^{\Z}_{\bC,Y_{min,min},c} \stackrel{\eqref{afdasdfqwefq}}{\simeq} i_{!}i^{*}K\cX^{\Z}_{\bC,Y_{min,min}} \stackrel{\eqref{qewfewdqdewdqewd}}{\simeq}
i_{!}i^{*} K\cX_{\bC_{Y}}^{\Z} \stackrel{\simeq}{\to} K\cX_{\bC_{Y}}^{\Z} $$
where the last equivalence   is an equivalence since $ K\cX_{\bC_{Y}}^{\Z} $ is already continuous.
\end{proof}

 We can now finish the proof of Theorem \ref{weitgowergerfrwferfw}.
 We have   equivalences  of morphisms
 $$\mu^{DL}_{ HK\cX^{\Z}_{\bC,Y_{min,min},c}} \stackrel{ {\scriptsize \cite[Cor. 8.25]{desc}  }}{\simeq}  \gamma_{K\cX^{\Z}_{\bC,Y_{min,min},c}} \stackrel{Prop. \ref{wiothgerththerh}}{\simeq}
 \gamma_{K\cX^{\Z}_{\bC_{Y} }}  \stackrel{{\scriptsize \cite[Cor. 8.25]{desc}  } }{\simeq}   \mu^{DL}_{HK\cX^{\Z}_{\bC_{Y}}}\ .$$ 
At the beginning of this proof we have seen that $ \mu^{DL}_{HK\cX^{\Z}_{\bC_{Y} }}$ is an equivalence. 
We conclude that
$\mu^{DL}_{ HK\cX^{\Z}_{\bC,Y_{min,min},c}}$ is an equivalence, too.
It follows that  $\Z\Sp\cX_{\loc}\langle DL\rangle$ contains $\Yo^{s}_{\loc}(Y_{min,min})$ for all $\Z$-sets $Y$, and this implies the assertion of the theorem.
\end{proof}
Corollary   \ref{wrthiojwergwergwerrwfg} now implies:
\begin{kor}\label{qerigojoqrfewfqewfqf}
We have
$\Z\Sp\cX_{\loc}\langle \disc\rangle \subseteq \PV_{K\bX_{\bC}^{\Z}}$.
\end{kor}

Let $A$ be a unital $C^{*}$-algebra with an action of $\Z$. We consider the $C^{*}$-category $\Hilb_{c}(A)$ of Hilbert $A$-modules  and compact operators which has the induced $\Z$-action \cite[Ex. 2.9]{cank}.
Then the square \eqref{qwdqwewqrwdqweddwdwdwd1eee} is cartesian by 
Corollary \ref{qerigojoqrfewfqewfqf}.   We let $\Z'$ denote the group which acts by automorphisms on $\Res^{\Z}(A)$ (via the identification of $\Z'$   with the original group $\Z$)
and on $\Z_{min,min}$ by translations. 
We let $\sigma$ denote the action by functoriality  of the generator of $\Z'$ on various derived objects. 
\begin{prop}\label{eorkjgwegreegrwegr9}
We have equivalences
\begin{enumerate}
\item \label{werjkngwerge}$K\cX^{\Z}_{\Hilb_{c}(A)}(\Z_{min,min})\simeq K(\Res^{\Z}(A))$
\item \label{werjkngwerge1} $K\cX^{\Z}_{\Hilb_{c}(A)}(\Z_{can,min})\simeq K(A\rtimes \Z)$
\item \label{werjkngwerge2}$K\cX^{\Z}_{\Hilb_{c}(A)}(\Z_{can,min}\otimes \Z_{min,min}) \simeq \Sigma K(\Res^{\Z}(A))$.
\end{enumerate}
Furthermore, the coarse PV-sequence
is equivalent to a fibre sequence
$$ \Sigma^{-1}K(A\rtimes \Z) \to  K(\Res^{\Z}(A))\stackrel{1-\sigma}{\to}  K(\Res^{\Z}(A))\ .$$
\end{prop}
\begin{proof}
In order to prepare the argument for the second assertion we will actually show that
there is an equivalence $$K\cX^{\Z}_{\Hilb_{c}(A)}(\Z_{min,min})\simeq K(\Res^{\Z}(A))$$ of spectra with an action of $\Z'$.
Using that   $\Z_{min,min}\cong \Ind_{1}^{\Z}(*)$, by \cite[Cor. 10.5.2]{coarsek}  have an equivalence \begin{equation}\label{qwfqqewfeqwfewddq}
K\cX^{\Z}_{\Hilb_{c}(A)}(\Z_{min,min})\stackrel{\simeq}{\to} K\cX_{\Res^{\Z}(\Hilb_{c}(A))}(*)\ .
\end{equation}
Unfolding \cite[Def. 4.19]{coarsek} we get  
an isomorphism $ \bV_{\Res^{\Z}(\Hilb_{c}(A))}(*)\cong\Res^{\Z}(\Hilb_{c}(A))^{u}  $, where $(-)^{u}$ denotes the operation of taking the full subcategory of unital objects. We  therefore get an equivalence 
\begin{equation}\label{fewqwfqewd}
K\cX_{\Res^{\Z}(\Hilb_{c}(A))}(*)\simeq \Kcat(\Res^{\Z}(\Hilb_{c}(A)^{u}))\ .
\end{equation} 
We consider $A$ as a $C^{*}$-category $\bA$ with a single object.  We then have a fully faithful functor
$\psi:\bA\to \Res^{\Z}( \Hilb_{c}(A)^{u})$ which sends the unique object of $\bA$ to the Hilbert $A$-module given by $A$ with the right  $A$-multiplication and the scalar product $\langle a,a'\rangle:=a^{*}a'$. 
 Since $A$ is assumed to be unital the latter object  of $\Hilb_{c}(A)$ is indeed unital. The functor 
 $\bA\to \Res^{\Z}( \Hilb_{c}(A)^{u})$
 is a Morita equivalence \cite[18.15]{cank}.
Since $\Kcat$ is Morita invariant and extends the $K$-theory functor for $C^{*}$-algebras
 we get  the  equivalence \begin{equation}\label{adfadsfdfasff}
K(\Res^{\Z}(A))\simeq \Kcat(\bA) \stackrel{\simeq}{\to}\Kcat(\Res^{\Z}(\Hilb_{c}(A)^{u}))\ .
\end{equation}
The equivalence in Assertion \ref{werjkngwerge} is the composition of the equivalences \eqref{qwfqqewfeqwfewddq}, \eqref{fewqwfqewd} and \eqref{adfadsfdfasff}.

   We let $\sigma$ denote the generator of $\Z'$. Then  for $n$ in $\Z_{min,min}$ we have $\sigma(n)=n+1$. We will show that the composition 
 $\eqref{fewqwfqewd}\circ  \eqref{qwfqqewfeqwfewddq}$ and the equivalence
 \eqref{adfadsfdfasff}  preserve the action of $\sigma$ up to equivalence.
  Using  the explicit description of
 the equivalence \eqref{qwfqqewfeqwfewddq} given in the proof of \cite[Prop. 10.1]{coarsek}
  the equivalence $\eqref{fewqwfqewd}\circ  \eqref{qwfqqewfeqwfewddq}$ is induced by the functor
 $$\phi:\bV^{\Z}_{\Hilb_{c}(A)}(\Z_{min,min})\to \bV_{\Res^{\Z}(\Hilb_{c}(A))}(*)\cong \Res^{\Z}(\Hilb_{c}(A))^{u}\ .$$
 The functor $\phi$
 sends the object  $(C,\rho,\mu)$ of $\bV^{\Z}_{\Hilb_{c}(A)}(\Z_{min,min})$ to the submodule $C(0):=\mu(\{0\})C$ in $\Hilb_{c}(A)^{u}$.
 In contrast to the general case considered in  \cite[Prop. 10.1]{coarsek}
 here we can take this preferred image of the projection $\mu(\{0\})$.
 The functor $\phi$ furthermore sends a  morphism $B:(C,\rho,\mu)\to (C',\rho',\mu')$ in $\bV^{\Z}_{\Hilb_{c}(A)}(\Z_{min,min})$   to
 $\mu'(\{0\})B\mu(\{0\}):C(0)\to C'(0)$.

 Note that $$\phi(\sigma(C,\rho,\mu))=\phi( (C,\rho,\sigma_{*}\mu))  =(\sigma_{*}\mu)(\{0\})C=\mu(\{-1\})C=:C(-1)\ .$$
 Using the invariance of $\mu$ we see that the unitary  multiplier $\rho_{\sigma}$ restricts to a unitary multiplier  
 $\rho_{\sigma,-1,0}:C(-1) \to  \sigma C(0)$. 
 We can therefore define a natural unitary isomorphism 
 $v:\phi\circ \sigma\to \sigma\circ \phi$ such that its evaluation at $(C,\rho,\mu)$ is given by $\rho_{\sigma,-1,0}$. Since $\Kcat$ sends unitarily equivalent functors to equivalent maps \cite[Lem. 17.11]{cank} we conclude that   $\eqref{fewqwfqewd}\circ  \eqref{qwfqqewfeqwfewddq}$ preserves the action of $\sigma$ up to equivalence.
  
In order to show that    \eqref{adfadsfdfasff} commutes with $\sigma$ up to equivalence  we construct a natural unitary isomorphism
$w:\psi\circ \sigma\to \sigma \circ \psi$. The evaluation of $w$ at the unique object of $\bA$ is the unitary isomorphism
${}^{\sigma}(-):A\to \sigma A$ of Hilbert $A$-modules, where
 $ \sigma A$ is the vector space $A$ with the new    Hilbert $A$-module described in \cite[Ex. 2.10]{cank}, and $a\mapsto {}^{\sigma}a$
 is the action of $\Z'$ on $ A $.

  Assertion 
   \ref{werjkngwerge1}  follows from \cite[Prop. 2.8.3]{coarsek} applied to $X=*$ and $G=\Z$.
   
  Assertion \ref{werjkngwerge2} follows from  Assertion \ref{werjkngwerge} 
  and the equivalence in
   \eqref{gwegwergrefrew}.
   
Finally, the last assertion is a consequence of Proposition \ref{werkgjowergwerfwerferfewrfwerf}
and the observation that the equivalence in Assertion 
   \ref{werjkngwerge} preserves the $\sigma$-actions up to equivalence.
\end{proof}


\section{What is in $\Z\Sp\cX_{\loc}\langle \disc\rangle$ }

In this section we again consider the example $E^{\Z}=K\cX^{\Z}_{\bC}$.
By Corollary \ref{qerigojoqrfewfqewfqf}
we have
$\Z\Sp\cX_{\loc}\langle \disc\rangle \subseteq \PV_{K\bX_{\bC}^{\Z}}$. Hence  we know that the  coarse PV-square   \eqref{qwdqwerrwqwdqweddwdwdwd1eee} is  cartesian for discrete $X$. In the present section  we show that 
$\Z\Sp\cX_{\loc}\langle \disc\rangle$ contains the motives of many non-discrete $\Z$-bornological coarse spaces.   The main result   is the following theorem.
\begin{theorem}\label{werigosetrrgertgerwgw}Assume one of the following:
\begin{enumerate}
\item\label{ijtgoertgegertgwee} $X$ has weakly finite asymptotic dimension.
\item \label{ijtgoertgegertgwee1} $X$ has bounded geometry and the coarse Baum-Connes assembly maps $\mu_{K\cX^{\Z}_{ \bC}(-\otimes \Z_{min,min}),X}$ and
$\mu_{K\cX^{\Z}_{\bC}(-\otimes \Z_{can,min}),X}$  are   equivalences.
\end{enumerate}
 Then $\Yo^{s}_{\loc}(X)\in \Z\Sp\cX_{\loc}\langle \disc\rangle$ and   
  the coarse PV-square   \eqref{qwdqwerrwqwdqweddwdwdwd1eee} is  cartesian.\end{theorem}
  \begin{proof}
  By     Corollary \ref{qerigojoqrfewfqewfqf} the  first part of the assertion  implies the second.
  
  Considering a bornological coarse space as a $\Z$-bornological coarse space we get a functor $\Res_{\Z}:\BC\to \Z\BC$.   A morphism in $ \Sp\cX$ is called a local equivalence
  if it is sent to an equivalence by the    functors  $K\cX^{\Z}_{ \bC}(-\otimes \Z_{min,min})$ and $K\cX_{ \bC}(-\otimes \Z_{min,min})$
   which are considered as non-equivariant homology theories.  
   As in the $\Z$-equivariant case we let $\ell:\Sp\cX\to \Sp\cX_{\loc}$ denote the localization at the local equivalences.
      We get a commutative diagram
   $$\xymatrix{\ar@/^-1cm/[dd]_{\Yo^{s}_{\loc}}\BC\ar[r]^{\Res_{\Z}}\ar[d]^{\Yo^{s}}&\ar@/^1cm/[dd]^{\Yo^{s}_{\loc}}\Z\BC\ar[d]^{\Yo^{s}}\\\Sp\cX\ar[r]^{\Res_{\Z}}\ar[d]^{\ell}&\Z\Sp\cX\ar[d]^{\ell}\\\Sp\cX_{\loc}\ar[r]^{\Res_{\Z}}&\Z\Sp\cX_{\loc}}$$

If $X $ in $\Sp\cX$ satisfies Assumption \ref{werigosetrgertgerwgw}.\ref{ijtgoertgegertgwee}, then $\Yo^{s}(X)\in \Sp\cX\langle \disc\rangle$ by  \cite[Thm 5.59]{buen}. This immediately implies that
$\Res_{\Z}(\Yo^{s}_{\loc}(X))\in  \Z\Sp\cX_{\loc}\langle \disc\rangle$.

We now assume that $X$ satisfies Assumption \ref{werigosetrgertgerwgw}.\ref{ijtgoertgegertgwee1}. Since the assertion of the  theorem only depends on the coarse equivalence class of $X$ we can assume that $X$
has strongly bounded geometry \cite[Def. 7.75]{buen}. Then for every coarse entourage $U$ of $X$ 
 the Rips complex $P_{U}(X)$ \cite[Ex. 2.6]{ass} is a finite-dimensional simplicial complex. The spherical path metric induces a bornological coarse structure  on the Rips complex such that  $X\to P_{U}(X)$ is an equivalence of bornological coarse spaces.
  Recall from \cite[Def. 9.7]{ass} that the universal  coarse assembly map is induced by the morphism \begin{equation}\label{adfqewfqewfe}
\mu_{\Yo^{s},X}:  \colim_{U\in \cC_{X}} \Yo^{s}(\cO^{\infty}(P_{U}(X)))\to  \Sigma \Yo^{s}(X)
\end{equation}
  derived from the cone sequence.  
  For any non-equivariant strong homology theory $F:\BC\to \bM$ the coarse assembly map $\mu_{F,X}$ from \eqref{fqwefwqedwqedqwedqewd} is then given by
  $\mu_{F,X}\simeq F( \mu_{\Yo^{s},X})$.

The Assumption \ref{werigosetrgertgerwgw}.\ref{ijtgoertgegertgwee1} together with Remark \ref{wrgijowergwregrwef}
  imply   that $ \mu_{\Yo^{s},X}$ is a local equivalence. Hence applying $\ell$ to  \eqref{adfqewfqewfe} we get the equivalence 
  \begin{equation}\label{adfqewfqewfe1}
\mu_{\Yo^{s}_{\loc},X}:  \colim_{U\in \cC_{X}} \Yo^{s}_{\loc}(\cO^{\infty}(P_{U}(X)))\stackrel{\simeq}{\to}  \Sigma \Yo^{s}_{\loc}(X)
\end{equation}
in $\Sp\cX_{\loc}$.
   The functor $ \Yo^{s}_{\loc}\circ \cO^{\infty}$ is homotopy invariant and excisive for closed decompositions of uniform bornological coarse spaces. Furthermore 
if $Y$ is a set, then $$\Yo^{s}_{\loc}( \cO^{\infty}(Y_{min,min,disc}))\simeq \Sigma \Yo^{s}_{\loc}(Y_{min,min})\in \Sp\cX_{\loc}\langle \disc\rangle\ .$$
 Using that $  \Sp\cX_{\loc}\langle \disc\rangle$ is a thick subcategory of $\Sp\cX_{\loc}$  and   $P_{U}(X)$ is finite-dimensional we
 can now use a finite induction by the skeleta of $X$
in order to conclude that  $$\Yo^{s}_{\loc}(\cO^{\infty}(P_{U}(X)))\in  \Sp\cX_{\loc}\langle \disc\rangle\ .$$ 
Since $  \Sp\cX_{\loc}\langle \disc\rangle$ even localizing  \eqref{adfqewfqewfe1} finally implies that  $$  \Yo^{s}_{\loc}(X)\in  \Sp\cX_{\loc}\langle \disc\rangle\ .$$
\end{proof}

\bibliographystyle{alpha}
\bibliography{forschung2021}

\begin{thebibliography}{BEKW20b}

\bibitem[BCKW]{unik}
U.~Bunke, D.-Ch. Cisinski, D.~Kasprowski, and Ch. Winges.
\newblock {Controlled objects in left-exact $\infty$-categories and the Novikov
  conjecture}.
\newblock \href{https://arxiv.org/abs/1911.02338}{arXiv:1911.02338}.

\bibitem[BEa]{cank}
U.~Bunke and A.~Engel.
\newblock {Additive $C^{*}$-categories and $K$-theory}.
\newblock \href{https://arxiv.org/abs/2010.14830}{arXiv:2010.14830}.

\bibitem[BEb]{coarsek}
U.~Bunke and A.~Engel.
\newblock {Topological equivariant coarse $K$-homology}.
\newblock \href{https://arxiv.org/abs/2011.13271}{arXiv:2011.13271}.

\bibitem[BE20a]{ass}
U.~Bunke and A.~Engel.
\newblock {Coarse assembly maps}.
\newblock {\em J.\ Noncommut.\ Geom.}, 14(4):1245--1303, 2020.

\bibitem[BE20b]{buen}
U.~Bunke and A.~Engel.
\newblock {\em Homotopy theory with bornological coarse spaces}, volume 2269 of
  {\em Lecture Notes in Math.}
\newblock Springer, 2020.
\newblock \href{https://arxiv.org/abs/1607.03657}{arXiv:1607.03657}.

\bibitem[BEKW20a]{coarsetrans}
U.~Bunke, A.~Engel, D.~Kasprowski, and C.~Winges.
\newblock Transfers in coarse homology.
\newblock {\em M{\"u}nster J.\ Math.}, 13:353--424, 2020.

\bibitem[BEKW20b]{equicoarse}
U.~Bunke, A.~Engel, D.~Kasprowski, and Ch. Winges.
\newblock {Equivariant coarse homotopy theory and coarse algebraic
  $K$-homology}.
\newblock In {\em {$K$-Theory in Algebra, Analysis and Topology}}, volume 749
  of {\em Contemp.\ Math.}, pages 13--104, 2020.

\bibitem[BEKW20c]{desc}
U.~Bunke, A.~Engel, D.~Kasprowski, and Ch. Winges.
\newblock {Injectivity results for coarse homology theories}.
\newblock {\em Proc.\ London Math.\ Soc.}, 121(3):1619--1684, 2020.

\bibitem[BELa]{bel-paschke}
U.~Bunke, A.~Engel, and M.~Land.
\newblock {Paschke duality and assembly maps}.
\newblock \href{http://arxiv.org/abs/2107.02843}{arxiv:2107.02843}.

\bibitem[BELb]{KKG}
U.~Bunke, A.~Engel, and M.~Land.
\newblock A stable $\infty$-category for equivariant $\mathrm{K\!K}$-theory.
\newblock \href{https://arxiv.org/pdf/2102.13372.pdf}{arxiv:2102.13372}.

\bibitem[Bla98]{blackadar}
B.~Blackadar.
\newblock {\em {$K$-Theory for Operator Algebras}}.
\newblock Cambridge University Press, 2nd edition, 1998.

\bibitem[Joa03]{joachimcat}
M.~Joachim.
\newblock {$K$-homology of $C^{\ast}$-categories and symmetric spectra
  representing $K$-homology}.
\newblock {\em Math. Ann.}, 327:641--670, 2003.

\bibitem[Kas88]{kasparovinvent}
G.~G. Kasparov.
\newblock {Equivariant $K\!K$-theory and the Novikov conjecture}.
\newblock {\em Invent.\ Math.}, 91(1):147--201, 1988.

\bibitem[Kra20]{kranz}
J.~Kranz.
\newblock An identification of the {B}aum-{C}onnes and {D}avis-{L}{\"u}ck
  assembly maps.
\newblock {\em M{\"u}nster J. of Math.}, 14:509--536, 2020.

\bibitem[LS16]{L_ck_2016}
Wolfgang L{\"u}ck and Wolfgang Steimle.
\newblock A twisted bass{\textendash}heller{\textendash}swan decomposition for
  the algebraic k-theory of additive categories.
\newblock {\em Forum Mathematicum}, 28(1), jan 2016.

\bibitem[PV80]{pv}
M.~{Pimsner} and D.~{Voiculescu}.
\newblock Exact sequences for $k$-groups and ext-groups of certain
  cross-product {$C^{*}$}- algebras.
\newblock {\em J. Oper. Theory}, 4:93--118, 1980.

\bibitem[Yu00]{yu_embedding_Hilbert_space}
G.~Yu.
\newblock {The coarse Baum--Connes conjecture for spaces which admit a uniform
  embedding into Hilbert space}.
\newblock {\em Invent. math.}, 139(1):201--240, 2000.

\end{thebibliography}

\end{document}